\documentclass{amsart}

\usepackage{amsmath,amssymb,amscd,url}
\title[Nilpotent covers and non-nilpotent subsets]{Nilpotent covers and non-nilpotent subsets of finite groups of Lie type}

\author{Azizollah Azad}
\address{Department of Mathematics, Faculty of Sciences, Arak University, Arak 38156-8-8349, Iran}
\email{a-azad@araku.ac.ir}

\author{John R. Britnell}
\address{Department of Mathematics, Imperial College London, South Kensington Campus, London SW7 2AZ, United Kingdom}
\email{j.britnell@imperial.ac.uk}

\author{Nick Gill}
\address{Escuela de Matem\'atica, Universidad de Costa Rica, 11501 San Jos\'e, Costa Rica}
\email{nickgill@cantab.net}

\date{8 August 2014}


\theoremstyle{plain}
\newtheorem{thm}{Theorem}[section]

\newtheorem{lem}[thm]{Lemma}
\newtheorem{prop}[thm]{Proposition}
\theoremstyle{definition}
\newtheorem{conj}[thm]{Conjecture}
\newtheorem{question}[thm]{Question}
\newtheorem{defn}[thm]{Definition}

\def\S{Section~}

\newcommand{\reduce}{\!}

\newcommand{\Fq}{\mathbb{F}_q}

\newcommand{\rank}{\mathrm{rank}}

\newcommand{\omegac}[1]{\omega_c({#1})}
\newcommand{\omegaone}[1]{\omega_1({#1})}
\newcommand{\omegan}[1]{\omega_\infty({#1})}

\newcommand{\zz}{\mathbb{Z}}
\newcommand{\longintegers}{\zz^+\cup\{\infty\}}

\newcommand{\GL}{\mathrm{GL}}
\newcommand{\PGL}{\mathrm{PGL}}
\newcommand{\SL}{\mathrm{SL}}
\newcommand{\PSL}{\mathrm{PSL}}

\newcommand{\GU}{\mathrm{GU}}
\newcommand{\PGU}{\mathrm{PGU}}
\newcommand{\SU}{\mathrm{SU}}
\newcommand{\PSU}{\mathrm{PSU}}
\newcommand{\Sp}{\mathrm{Sp}}

\newcommand{\Hom}{\mathrm{Hom}}

\newcommand{\N}{\mathbb{N}}

\newcommand{\R}{\mathbb{R}}
\newcommand{\C}{\mathbb{C}}

\begin{document}

\begin{abstract}
Let~$G$ be a finite group and~$c$ an element of~$\longintegers$. A subgroup~$H$ of~$G$ is said to be {\it $c$-nilpotent}
if it is nilpotent and has nilpotency class at most~$c$. A subset~$X$ of~$G$ is said to be
{\it non-$c$-nilpotent} if it contains no two elements~$x$ and~$y$ such that the subgroup ${\langle x,y\rangle}$
is $c$-nilpotent. In this paper we study the quantity~$\omegac{G}$, defined to be the size of the largest non-$c$-nilpotent subset of~$L$.

In the case that~$L$ is a finite group of Lie type, we identify covers of~$L$ by $c$-nilpotent subgroups, and
we use these covers to construct large non-$c$-nilpotent sets in~$L$. We prove that for groups $L$ of fixed rank $r$, there exist constants
$D_r$ and $E_r$ such that ${D_r N \leq \omega_\infty(L) \leq E_r N}$, where $N$ is the number of maximal tori in $L$.

In the case of groups~$L$ with twisted rank~$1$, we provide exact formulae for~$\omegac{L}$ for all $c\in\longintegers$. If we write $q$ for the level of the Frobenius endomorphism associated with $L$ and assume that $q>5$, then $\omegan{L}$ may be expressed as a polynomial in $q$ with coefficients in $\{0,1\}$.
\end{abstract}

\maketitle

\section{Introduction and results}

Let~${G}$ be a group and~${c}$ an element of~${\longintegers}$. For ${c\in\mathbb{Z}^+}$,
we define~${G}$ to be {\it ${c}$-nilpotent} if~${G}$ is nilpotent of class at most~${c}$. We define~${G}$
to be {\it ${\infty}$-nilpotent} if~${G}$ is nilpotent.

\subsection{Non-${c}$-nilpotent subsets}

A subset~${X}$ of~${G}$ is said to be {\it non-${c}$-nilpotent} if, for any two distinct elements~${x}$ and~${y}$ in~${X}$,
the subgroup ${\langle x, y\rangle}$ of~${G}$ which they generate is not ${c}$-nilpotent.
The subset~${X}$ is said to be {\it non-nilpotent} if it is non-${\infty}$-nilpotent.

Define~${\omegac{G}}$ to be the maximum order of a non-${c}$-nilpotent subset of~${G}$.
It is a trivial observation that for any group~${G}$, we have
\[ \omega_1(G)\geq \omega_2(G)\geq \cdots \geq \omega_\infty(G). \]
Furthermore, if the class of a nilpotent subgroup of~${G}$ is bounded above by an integer~${d}$, then clearly
\[ \omega_d(G)=\omega_{d+1}(G) = \cdots = \omegan{G}. \]
Certainly such a bound exists whenever~${G}$ is finite. Furthermore if~${G}$
is a nilpotent group of class~${d}$, then ${\omegac{G}=1}$ for ${c\geq d}$, and for ${c=\infty}$.

Our main interest in this paper will be the quantity~${\omegan{G}}$, where~${G}$ is a finite group of Lie type. However we shall also
consider~${\omegac{G}}$ for various finite values of~${c}$. In particular the case when ${c=1}$ is of interest, since
${\omegaone{G}}$ is the maximum order of a non-commuting subset (i.e.\ a subset~${X}$ of~${G}$ such that ${[x,y]\neq 1}$
for all~${x,y\in X}$).

\subsection{Nilpotent covers}

There is a close connection between the non-${c}$-nilpotent subsets of maximum order in a group~${G}$, and {\it ${c}$-nilpotent
covers} of~${G}$. Let~${\mathcal{N}}$ be a family of ${c}$-nilpotent subgroups of~${G}$. We shall be interested in two possible
properties of~${\mathcal{N}}$:
\begin{itemize}
\item {\bf Covering:} If for every ${g\in G}$ there exists ${X\in\mathcal{N}}$ such that ${g\in X}$, then we say that~${\mathcal{N}}$
is a {\it ${c}$-nilpotent cover} of~${G}$, or that~${\mathcal{N}}$ {\it covers}~${G}$. If ${c=\infty}$, we say that~${\mathcal{N}}$
is a {\it nilpotent cover} of~${G}$.
\item {\bf\boldmath${2}$-minimality:} If for every subgroup ${X_i\in \mathcal{N}}$ there is an element~${g_i}$
(called a {\it distinguished element}) such that ${i\neq j}$ implies that ${\langle g_i, g_j\rangle}$ is not
${c}$-nilpotent, then we say that~${\mathcal{N}}$ is {\it ${2}$-minimal}.
\end{itemize}

We remark that every group admits a ${c}$-nilpotent cover for all \mbox{${c\in\longintegers}$}, since the family consisting of
all cyclic subgroups is one such. But not every group admits a ${2}$-minimal ${c}$-nilpotent cover: examples are
the symmetric groups~${S_n}$ for~${n\geq 15}$ in the case ${c=1}$ (\cite{brown1, brown2}).

Suppose that ${\mathcal{N}}$ is a ${2}$-minimal ${c}$-nilpotent cover, and that ${X_1,X_2\in\mathcal{N}}$ contain distinguished elements~$g_1$ and $g_2$ respectively. We note that if ${X_2\neq X_1}$, then~${g_1\notin X_2}$ (otherwise the subgroup  ${\langle g_1,g_2\rangle}$ would be a subgroup of ${X_2}$, and hence ${c}$-nilpotent; this
contradicts the condition for ${2}$-minimality). It follows that the removal from~${\mathcal{N}}$ of any one of its members
results in a family which is not a cover of~${G}$ (since the distinguished element of the removed member is not in
any other member). Hence a ${2}$-minimal ${c}$-nilpotent cover of~${G}$ is, in particular, a {\it minimal} ${c}$-nilpotent cover.
The converse is false, however: it is not necessarily true that a minimal ${c}$-nilpotent cover of~${G}$ is ${2}$-minimal. This
is clear from the fact that there exist finite groups with no ${2}$-minimal ${c}$-nilpotent covers.

The significance of these properties to the calculation of ${\omega_c(G)}$, and in particular the value of finding a ${2}$-minimal cover of ${G}$, is shown by the following proposition.
\begin{prop}\label{p: N suffices}
Let ${\mathcal{N}}$ be a family of ${c}$-nilpotent subgroups of ${G}$.
\begin{enumerate}
\item If ${\mathcal{N}}$ covers ${G}$, then ${|\mathcal{N}|\ge \omega_c(G)}$.
\item If ${\mathcal{N}}$ is ${2}$-minimal, then ${|\mathcal{N}|\le \omega_c(G)}$.
\item If ${\mathcal{N}}$ is a ${2}$-minimal cover of ${G}$, then ${|\mathcal{N}|=\omega_c(G)}$.
\end{enumerate}
\end{prop}

\begin{proof}
For the first part, suppose ${X}$ is a set of size greater than ${|\mathcal{N}|}$. If ${\mathcal{N}}$ covers~${G}$, then there exist two elements of ${X}$
lying in one member of ${\mathcal{N}}$, and so ${X}$ cannot be non-${c}$-nilpotent. The second part is obvious, since the distinguished elements of a ${2}$-minimal family form a non-${c}$-nilpotent set. The third part of the proposition follows immediately from the other two.
\end{proof}

Lemmas~\ref{lem: upper bound} and~\ref{lem: lower bound} below provide useful generalizations of the inequalities in
the first two parts of Proposition \ref{p: N suffices}.

\subsection{Results and structure}

The paper is structured as follows. In \S\ref{s: lag} we give some background results on linear algebraic groups.
In \S\ref{s: groups} we set out the group-theoretic notation we will use throughout the paper, and we prove a number
of basic lemmas pertaining to non-${c}$-nilpotent groups.

In \S\ref{s: lower} we consider the case that where ${L=G^F}$, where ${G}$ is a simple linear algebraic group of rank ${r}$, and ${F}$ is a Frobenius endomorphism of~${G}$. Theorem~\ref{thm: putting sets together} offers a general method for producing lower bounds for~${\omegan{L}}$, by counting the Sylow subgroups of ${G}$ for certain prime divisors of ${|G|}$.  In the course of proving this theorem, we provide in Lemma \ref{lem: regular unipotent set} a generalization and strengthening of the main result of \cite{azad}.

In \S\ref{s: lower} we also prove the following theorem, which is the main result of this paper.
\begin{thm}\label{thm: rank-dependent lower bound}
For every ${r>0}$, there exist constants ${D_r, E_r>0}$ such that for any simple linear algebraic group ${G}$ of rank ${r}$, and for any Frobenius endomorphism~${F}$ of ${G}$, we have
\[ D_r N(G^F) \leq \omega_\infty(G^F) \leq E_r N(G^F), \]
where ${N(G^F)}$ is the number of ${F}$-stable maximal tori in ${G}$.
\end{thm}

The number ${N(G^F)}$ is given by a result of Steinberg (Proposition \ref{p: steinberg} below) as ${q^{|\Phi|}}$, where ${q}$ is the level of the Frobenius enomorphism ${F}$, and ${\Phi}$ is a root system for ${G}$. In fact it would be possible, in the statement of Theorem \ref{thm: rank-dependent lower bound}, to take ${N(G^F)}$ instead to be the number of maximal tori in ${G^F}$ itself; this point is discussed in the proof of Proposition \ref{p: maxtori} below.

It seems reasonable to conjecture that an upper bound of the same form exists not just for ${\omega_\infty(G^F)}$ but for ${\omega_1(G^F)}$, and more speculatively, that the dependence on rank can be removed.
We may formulate the following conjecture.
\begin{conj}
There exist absolute constants ${D, E>0}$ such that, for any simple linear algebraic group ${G}$ and any Frobenius endomorphism ${F}$ of ${G}$, we have
\[
DN(G^F) \leq \omega_\infty(G^F) \leq \cdots \leq \omega_1(G) \leq E N(G^F),
\]
where ${N(G^F)}$ is the number of ${F}$-stable maximal tori in ${G}$.
\end{conj}

Theorem \ref{thm: rank-dependent lower bound} confirms the existence of a rank-dependent constant ${D}$.
In particular the existence of a constant ${D}$ is confirmed for groups of bounded rank, which includes all the exceptional groups.
The existence of a non-rank-dependent constant ${E}$ for the groups ${\PGL_n(q)}$ follows from the main result of \cite{aips}; and we note that the full
conjecture is also confirmed in the case of groups of twisted rank ${1}$ by the results of \S\ref{s: exact} of this paper.

In \S\ref{s: exact} we assume that~${G^F}$ is a finite group of Lie type where ${G}$ is simple and the twisted rank of ${G^F}$ is~${1}$. We prove that, for all
${c\in\longintegers}$, the group~${G^F}$~admits a ${2}$-minimal ${c}$-nilpotent cover. Furthermore we construct explicit examples of such covers and we calculate their order, thereby producing exact formulae for~${\omegac{G^F}}$ in each case.

The nature of these formulae is somewhat remarkable, and we discuss some of their characteristics in \S\ref{s: questions}. We suggest a number of questions and conjectures arising from these observations, and from other results in the paper.

\subsection{Background results}

The value of~${\omegan{G}}$ has been studied for various groups; it has usually been denoted
${\omega(\mathcal{N}_G)}$. Endimioni has proved that if a finite group~${G}$ satisfies ${\omegan{G}\leq 3}$, then~${G}$
is nilpotent, while if ${\omegan{G}\leq 20}$, then~${G}$ is solvable; furthermore these
bounds cannot be improved \cite{Endimioni}. Tomkinson has shown that
if~${G}$ is a finitely generated solvable group such that ${\omegan{G}=n}$, then
${|G/Z^*(G)|\leq n^{n^4}}$,
where~${Z^*(G)}$ is the hypercentre of~${G}$ (\cite{Tomkinson}). Also, for a finite
non-solvable group~${G}$, it has been proved by the first author and Hassanabadi that~${G}$ satisfies the
condition ${\omegan{G}=21}$ if and only if ${G/Z^*(G)\cong A_5}$
(see \cite[Theorem 1.2]{am}).

In addition to the results in \cite{am}, the computation of~${\omegan{G}}$ for particular classes of groups~${G}$ has recently
started to garner attention. In particular, the first author has given lower bounds for~${\omegan{G}}$ (\cite{azad}) when ${G=\GL_n(q)}$.
In a forthcoming paper by the second and third authors, a nilpotent cover of ${\GL_n(q)}$ is constructed which is ${2}$-minimal when ${q>n}$;
this construction will establish the exact value of~${\omega_\infty(\mathrm{GL}_n(q)})$ for almost all values of ${n}$ and ${q}$ (see~\cite{bg}), and an upper bound in the remaining cases.

The particular statistic ${\omegaone{G}}$ has attracted considerable recent attention, and been calculated for various groups~${G}$; much of this work has
concentrated on the case of almost simple groups ${G}$ (\cite{aamz, aips, ap, brown1, brown2}).

The study of non-commuting sets in a group~${G}$, including the study of~${\omegaone{G}}$, goes back many years. In
1976, B.\,H.~Neumann famously answered the following question of Erd\H{o}s from a few years earlier: if all non-commuting
sets in a group~${G}$ are finite, does there exist an upper bound ${n=n(G)}$ for the size of a non-com\-muting set in~${G}$ (i.e.\ is~${\omegaone{G}}$ finite)? Neumann answered this question affirmatively by showing that if all non-com\-muting sets in a group~${G}$ are finite, then ${|G:Z(G)|}$ is finite
\cite{neumann}. Pyber subsequently gave a strong upper bound for~${|G:Z(G)|}$, subject to the same condition on ${G}$ (\cite{pyber}).

Related to this area of study is the problem of calculating, for a finite ${2}$-generator group ${G}$, the size ${\mu(G)}$ of the largest subset
${X\subseteq G}$ such that any pair of elements of ${X}$ generate ${G}$. Such sets are closely related to covers of ${G}$ by proper subgroups (in just the same way that we have seen that non-${c}$-nilpotent sets are related to ${c}$-nilpotent covers). The statistic ${\mu(G)}$ has been studied for the symmetric and alternating groups in \cite{blackburn}, and for the groups ${\GL_n(q)}$ and ${\SL_n(q)}$ in \cite{BEGHM}.

\subsection{Connections to perfect graphs}

We note a connection with a generalization of the commuting graph of a finite group ${G}$. The commuting graph of ${G}$ is the graph ${\Gamma_1(G)}$ whose vertices are the elements of ${G}$, with an edge joining vertices ${x}$ and ${y}$ if and only if ${x}$ and ${y}$commute in ${G}$. (An alternative definition excludes
central elements of ${G}$; the distinction is unimportant here.)
There is an obvious correspondence between the maximal abelian subgroups of ${G}$, and the maximal cliques in the commuting graph ${\Gamma_1(G)}$. From this fact it follows that the minimal size of a covering of ${G}$ by abelian subgroups is equal to the \emph{clique cover number} of ${\Gamma_1(G)}$, i.e.\ the minimal number of cliques required to cover its vertices.

Suppose that ${G}$ has a ${2}$-minimal abelian cover. Then the set of distinguished elements form an independent set in ${\Gamma_1(G)}$, and it follows that the clique cover number and the independence number (being the maximal size of an independent set of vertices) are the same for ${\Gamma_1(G)}$. (It is obvious that the clique cover number is at least as big as the independence number.)

A graph is \emph{perfect} if the clique cover number and the independence number coincide for every induced subgraph. It appears that it is not known which finite groups have perfect commuting graphs.\footnote{Peter Cameron has recently discussed this problem on his blog, at \url{http://cameroncounts.wordpress.com/2011/02/01/perfectness-of-commuting-graphs/}.}

There is an obvious generalization of ${\Gamma_1(G)}$ as follows: for ${c\in\longintegers}$ define ${\Gamma_c(G)}$ to be the graph whose vertices are the elements of ${G}$, with an edge joining vertices ${x}$ and ${y}$ if and only if the subgroup ${\langle x,y\rangle}$ is ${c}$-nilpotent. We can generalize the earlier observation connecting abelian covers to properties of ${\Gamma_1(G)}$ in the following way.

\begin{prop}\label{p: graph theory}
A finite group ${G}$ has a ${2}$-minimal ${c}$-nilpotent cover~${\mathcal{N}}$ if and only if the clique cover number and the independence number of ${\Gamma_c(G)}$ coincide.
\end{prop}
\begin{proof}
Suppose that ${G}$ has a ${2}$-minimal ${c}$-nilpotent cover~${\mathcal{N}}$. It is clear that the clique cover number can be no larger than ${|\mathcal{N}|}$. The first implication therefore follows from the observation that the set of distinguished elements in members of ${\mathcal{N}}$ forms an independent set of size ${|\mathcal{N}|}$.

Conversely suppose that $\mathcal{N}$ is a $c$-nilpotent cover. If $I$ is an independent set of cardinality $\mathcal{N}$, then each element of $\mathcal{N}$ must contain a unique element of $I$, and so $\mathcal{N}$ is ${2}$-minimal.
\end{proof}

Note that Proposition~\ref{p: graph theory} implies that a necessary condition for ${\Gamma_c(G)}$ to be perfect is that ${G}$ admits a ${2}$-minimal ${c}$-nilpotent cover.

We remark that, when ${c>1}$, it is not necessarily true that the maximal cliques in the resulting graph ${\Gamma_c(G)}$ correspond to maximal ${c}$-nilpotent subgroups of~${G}$; for instance, there exist two ${3}$-nilpotent groups of order ${64}$ for which the graph~${\Gamma_2(G)}$ has maximal cliques of size ${40}$. Thus the proof of Proposition~\ref{p: graph theory} yields some extra information: we can conclude that if~${G}$ has a ${2}$-minimal ${c}$\hbox{-}nil\-potent cover ${\mathcal{N}}$, then ${\Gamma_c(G)}$ admits a minimal covering by (not necessarily maximal) cliques such that each clique corresponds to a maximal ${c}$-nilpotent subgroup of ${G}$.

In a different direction, one might define the graph ${\Gamma_c(G)}$ to be {\it predictable} if all maximal cliques in ${\Gamma_c(G)}$ correspond to
maximal ${c}$-nilpotent subgroups of ${G}$. One can then ask for which groups ${G}$, and which values of ${c}$, is the graph ${\Gamma_c(G)}$ predictable. In particular is it true that if ${G}$ is a simple group, then ${\Gamma_c(G)}$ will be predictable for all ${c\in\longintegers}$?

\mbox{}

\section{Background on linear algebraic groups}\label{s: lag}

The material in this section is drawn primarily from \cite{Carter}. Throughout the section, ${G}$ is a connected
reductive linear algebraic group over an algebraically closed field~${K}$ of characteristic ${p>0}$.

Since~${G}$ is a linear algebraic group, we can write ${G\leq \GL_n(K)}$ for some integer~${n}$. An element ${g\in G}$ is then
said to be {\it semisimple} if~${g}$ is diagonalizable in~${\GL_n(K)}$, and it is said to be {\it unipotent} if its
only eigenvalue is~${1}$. It is a fact that the condition for~${g}$ to be semisimple (respectively unipotent) is independent
of the embedding of~${G}$ into ${\GL_n(K)}$, and so we can say that~${g}$ is semisimple (respectively unipotent)
element without reference to any particular embedding.

A {\it unipotent subgroup} of~${G}$ is a closed subgroup, all of whose elements are unipotent. A {\it Borel subgroup}
of~${G}$ is a maximal connected closed solvable subgroup of~${G}$. Borel subgroups always exist, and all Borel subgroups
of~${G}$ are conjugate \cite[p.\,16]{Carter}. We can write ${B=UT}$, where~${U}$ is the unipotent radical of~${B}$
(i.e.\ the maximal connected normal unipotent subgroup of~${B}$) and~${T}$ is a maximal torus of~${B}$ (and hence of~${G}$).
Note that the union of all Borel subgroups of~${G}$ is the whole of~${G}$, and in particular,
any semisimple element of~${G}$ lies in a maximal torus of~${G}$.

The {\it rank} of~${G}$, written ${\rank(G)}$, is defined to be the dimension of a maximal torus in~${G}$. We write~${W}$ for the {\it Weyl group} of~${G}$ and recall that if ${T}$ is a maximal torus of ${G}$, then ${N_G(T)/T\cong W}$.

\subsection{Regular elements and centralizers}

Any unattributed statements in this section can be found in \cite[\S1.14]{Carter}. We note first that
${\dim C_G(g)\geq \rank(G)}$ for all ${g\in G}$, where ${C_G(g)}$ denotes the centralizer of~${g}$ in~${G}$. An element~${g}$
of~${G}$ is said to be \textit{regular} if ${\dim C_G(g)=\rank(G)}$. If~${g}$ is regular in~${G}$, then~${C_G(g)^0}$, the
connected component of~${C_G(g)}$, is commutative.

The key facts for our purposes concerning regular semisimple elements are given by the following result from \cite{steinberg3}.
\begin{prop}\label{p: regular semisimple}
Let~${G}$ be a semisimple group and let ${s\in G}$ be semisimple. Then the following conditions are equivalent.
\begin{enumerate}
\item ${s}$ is regular.
\item ${C_G(s)^0}$ is a maximal torus in~${G}$.
\item ${s}$ is contained in a unique maximal torus in~${G}$.
\end{enumerate}
\end{prop}

Every connected reductive group~${G}$ contains regular unipotent elements and any two are conjugate in~${G}$. Let
${u\in G}$ be regular and unipotent; then every semi\-simple element in~${C_G(u)}$ lies in~${Z(G)}$, the centre of~${G}$
\cite[Proposition~5.1.5]{Carter}. The following result will be useful \cite[Proposition~5.1.3]{Carter}.

\begin{prop}\label{p: regular unipotent}
Let~${G}$ be a connected reductive group and let ${u\in G}$ be unipotent. Then the following conditions on~${u}$ are equivalent.
\begin{enumerate}
\item ${u}$ is regular.
\item ${u}$ lies in a unique Borel subgroup of~${G}$.
\end{enumerate}
\end{prop}

The final proposition of this section follows from \cite[Theorem 1.12.5]{gls3}.
\begin{prop}\label{prop: t divides W}
Let~${G}$ be a simple algebraic group, and let ${W}$ be the Weyl group of~${G}$. If~${t}$ is a prime divisor of~${|Z(G)|}$, then~${t}$ divides~${|W|}$.
\end{prop}

\subsection{Groups of Lie type}

A \emph{Frobenius endomorphism} of a reductive linear algebraic group ${G}$ is an algebraic endomorphism ${F}$ of ${G}$ with the property that there is an algebraic group embedding of ${G}$ into ${\GL_n(K)}$, for an algebraically closed field ${K}$ and ${n\in\N}$, and a power of ${F}$, say ${F^m}$, that is equal to the map on ${G}$ induced by a Galois automorphism of the field ${K}$. Let ${q_0}$ be the size of the fixed field of~${F^m}$. Then the \emph{level} of~${F}$ is defined to be
${q_0^{1/m}}$ (see \cite[Definition 2.1.9]{gls3}).

In what follows, we shall usually adopt the convention that ${q}$ is the level of the Frobenius map ${F}$. When dealing with the Ree groups ${{}^2B_2}(q)$ and ${{}^2G_2}(q)$
in Section~\ref{s: exact}, however, it will be convenient to take ${q}$ to be the square of the level of ${F}$ (which is ${q_0}$ in the notation above).

A Frobenius endomorphism is always an automorphism of ${G}$, considered as an abstract group, though the inverse map need not be algebraic.
We write ${G^F}$ for the fixed points of ${G}$ under ${F}$.
A {\it finite group of Lie type} is a group of the form~${G^F}$ where~${G}$ is and~${F}$
is a Frobenius endomorphism of~${G}$. We define the \emph{rank} of $G^F$ to equal the rank of $G$.

The {\it simple groups of Lie type} are defined in \cite[Definition 2.2.8]{gls3}; they are subgroups of~${G^F}$,
where~${G}$ is a simple linear algebraic group and~${F}$ is a Frobenius endomorphism.
A consequence of this terminology is that a simple group of Lie type is not necessarily a finite group of Lie type. For
the purposes of studying non-nilpotent sets, however, the next result renders the distinction irrelevant
(see the discussion following Lemma~\ref{lem: quasisimple}). 

\begin{lem}\label{lem: versions}
Let~${L}$ be a simple group of Lie type. Then ${L\cong G^F/Z(G^F)}$ for some simple linear algebraic group~${G}$ and
Frobenius endomorphism~${F}$.
\end{lem}
\begin{proof}
This follows from \cite[Theorem 2.2.6\,(f)]{gls3} and \cite[Definition 2.5.10]{gls3}.
\end{proof}

If~${F}$ is a Frobenius endomorphism of the group~${G}$, and if~${H}$ is a closed subgroup of~${G}$,
then we write~${H^F}$ for the set of points of~${H}$ fixed by~${F}$.

\subsection{Borel subgroups and tori in~${G^F}$}\label{s: tori intro}

We have already mentioned Borel subgroups and tori in the group~${G}$. We extend this notion to the group~${G^F}$ as
in \cite[\S1.18]{Carter}. Note first that a subgroup~${H}$ of~${G}$ which satisfies ${F(H)=H}$ is said to be {\it ${F}$-stable}.

A {\it Borel subgroup} of~${G^F}$ is a subgroup of the form~${B^F}$, where~${B}$ is an ${F}$-stable Borel subgroup of~${G}$.
Any two Borel subgroups of~${G^F}$ are conjugate in~${G^F}$; what is more, if we write ${B=UT}$ as above, then
\cite[(69.10)]{cr2} implies that ${N_{G^F}(U^F)=B^F}$.

A {\it maximal torus} of~${G^F}$ is a subgroup of the form~${T^F}$, where~${T}$ is an ${F}$-stable maximal torus of~${G}$.
An ${F}$-stable maximal torus of~${G}$ is said to be {\it maximally split} if it lies in an ${F}$-stable Borel subgroup
of~${G}$, and a maximal torus of~${G^F}$ is said to be {\it maximally split} if it has the form~${T^F}$ for some maximally
split torus~${T}$ in~${G}$. Any two maximally split tori of~${G^F}$ are conjugate in~${G^F}$.

It follows from the Lang-Steinberg theorem that the group $G$ contains an $F$-stable Borel subgroup and that any $F$-stable Borel subgroup contains an $F$-stable maximal torus of $G$. Thus $G^F$ contains both a Borel subgroup and a maximally split maximal torus.

The following result of Steinberg gives a precise count for the number of $F$-stable maximal tori in $G$ (see \cite[Theorem 3.4.1]{Carter}; we remark that the result as stated there applies also to the Ree and Suzuki groups). Recall that $q$ is defined to be the level of the Frobenius endomorphism $F$.

\begin{prop}\label{p: steinberg}
Let ${\Phi}$ be the root system for ${G}$. The number of ${F}$-stable maximal tori in~${G}$ is~${q^{|\Phi|}}$.
\end{prop}

We have already introduced the notation ${N(G^F)}$ for the number of ${F}$-stable maximal tori in ${G}$ (see Theorem~\ref{thm: rank-dependent lower bound}). Note that if $q$ is sufficient large, then the number of maximal tori in~${G^F}$ is equal to the number of ${F}$-stable maximal tori in ${G}$ (see \cite[Proposition 3.6.6]{Carter}).

\subsection{Sylow~${t}$-subgroups in~${G^F}$}\label{s: sylowt}

We are interested in the Sylow structure of the group~${G^F}$, for which we refer to~\cite{gls3,springersteinberg}.
Note first the elementary
facts that an element ${g\in G^F}$ is unipotent if and only if it has order ${{p^\ell}}$ for some ${\ell\geq 0}$, while~${g}$ is semisimple
if and only if~${p}$ does not divide the order of~${g}$.

In order to study the structure of a Sylow ${p}$-subgroup, we consider an ${F}$-stable Borel subgroup~${B}$ of~${G}$.
If we write ${B=UT}$ as above, then since~${B}$,~${U}$ and~${T}$ are all ${F}$-stable, we have ${B^F=U^F\rtimes T^F}$. Now
\cite[Theorem 2.3.4]{gls3} implies that~${U^F}$ is a Sylow ${p}$-subgroup of~${G^F}$.

We shall require the following fact, which follows from \cite[II.5.19]{springersteinberg}.

\begin{prop}\label{p: springersteinberg}
Let~${t}$ be a prime dividing~${|G^F|}$, with ${t\neq p}$, and suppose that~${X}$ is a Sylow ${t}$-subgroup of~${G^F}$.
Then~${X}$ lies inside~${N_G(T)}$, where~${T}$ is an ${F}$-stable
maximal torus of~${G}$. If~${t}$ is coprime to~${|W|}$, then~${X}$ lies inside~${T}$ itself
(and so, in particular,~${X}$ lies inside~${T^F}$, a maximal torus of~${G^F}$).
\end{prop}

Let ${x}$ and ${n}$ be positive integers. A \emph{primitive prime divisor} (also sometimes called a \emph{Zsigmondy prime}) of ${x^n-1}$ is a prime divisor ${s}$ which
does not divide ${x^k-1}$ for any ${k<n}$. A well-known theorem of Zsigmondy \cite{zsig} (in part also due to Bang \cite{bang}) states that if ${x>1}$ and ${n>2}$, then there
exists a primitive prime divisor of ${x^n-1}$ except in the case ${(x,n)=(2,6)}$. A primitive prime divisor exists also when ${n=2}$, unless
${x=2^b-1}$ for some~${b}$.

We shall need the following simple inequality.

\begin{lem}\label{lem: fermat}
Let ${p}$ be prime, and let ${q=p^a}$ for some ${a\in\N}$. Let ${n>1}$ be such that a primitive prime divisor of ${q^n-1}$ exists. Then
there is a primitive prime divisor ${r}$ of ${q^n-1}$ such that ${r > an}$.
\end{lem}

\begin{proof}
It is clear from Zsigmondy's theorem that if a primitive prime divisor of~${q^n-1}$ exists, and if~${q^n\neq 64}$, then a primitive prime divisor
to the base~${p}$ also exists; that is, a divisor ${r}$ of ${p^{an}-1}$ which does not divide ${p^b-1}$ for any~${b<an}$.
Now Euler's theorem states that ${p^{\phi(r)}\equiv 1\bmod r}$, where
${\phi}$ is the totient function. It follows that ${\phi(r)\ge an}$, and hence that ${r>an}$. It is easy to check that the result holds also for the exceptional case in Zsigmondy's theorem, arising when ${q^n=64}$.
\end{proof}

\subsection{Classical groups}\label{s: classical}

For certain finite groups of Lie type $G^F$ it is the case that $G^F$ is isomorphic to a central quotient, or a central extension, of a group of similarities of
a non-degenerate or zero form $\kappa$, on a vector space $V$ over a field $\mathbb{F}_{q^{\alpha}}$. These groups $G^F$ are the {\it classical groups of Lie type}.

Table~\ref{t: weyl} lists the relevant families, with the order of the root system $\Phi$, the order of the associated Wey group $W$, the name of the associated isometry group, and the value of $\alpha$. We include restrictions on the values of $q$ and $n$ in order to reduce the number of groups $G^F$ that are considered more than once. 
Note that the dimension of the vector space $V$ is equal to the subscript on the isometry group name; we refer to this number as the {\it dimension} of the group $G^F$.

\renewcommand{\arraystretch}{1.2}
\begin{table}[t]
\begin{center}
\begin{tabular}{cccccc}
\hline
${\Phi}$ & ${|\Phi|}$ & ${|W|}$ & ${\rm Isom}(\kappa)$ & $\alpha$ &
\\ \hline
${A_n}$ & ${n(n-1)}$ & ${n!}$ & $\begin{array}{c} \GL_{n+1}(q) \\ \GU_{n+1}(q) \end{array}$ & $\begin{array}{c} 1 \\ 2 \end{array}$ & $\begin{array}{c} \\  n\geq 2 \end{array}$\\ \hline
 ${B_n}$ & ${2n^2}$ & ${2^nn!}$ & ${\rm O}_{2n+1}(q)$ & 1 & $q$ odd, $n\geq 3$\\ \hline
${C_n}$ & ${2n^2}$ & ${2^nn!}$ & ${\Sp_{2n}(q)}$ & 1 & $n\geq 2$
\\ \hline
${D_n}$ & ${2n(n-1)}$ & ${2^{n-1}n!}$ & $\begin{array}{c} {\rm O}^+_{2n}(q) \\ {\rm O}^-_{2n}(q)\end{array}$ & 1 &  $\begin{array}{c} n\geq 4 \\  n\geq 4 \end{array}$
\\ \hline \\
\end{tabular}
\caption{\label{t: weyl}The families of groups~${L}$ of classical type.}
\end{center}
\end{table}

In the proof of Theorem~\ref{thm: rank-dependent lower bound}, we shall use properties of particular families of maximal tori in the classical groups $G^F$. We call these maximal tori {\it distinguished} and we describe them here in terms of the associated formed space $V$. We write $L$ for the group of similarities that is either a central quotient or central extension of $G^F$.

Suppose that $G^F$ is of type {\bf $A_n$}, {\bf $C_n$}, {\bf ${^2D_n}$} or {\bf ${{}^2A_{2m}}$}. Then it is well known (see, for instance \cite{bereczky}), that $L$ contains an irreducible cyclic subgroup. Let $X$ be of maximal order amongst such subgroups (a {\it Singer cycle}), and let $g$ be a generator of $X$. The irreducibility of $X$ guarantees that $g$ has distinct eigenvalues. Write $h$ for the projection (or for a lift) of $g$ in  $G^F$. Then $h$ is regular, and so by Proposition~\ref{p: regular semisimple} $h$ lies in a unique maximal torus $T^F$ of $G^F$. The conjugates of $T^F$ in $G^F$ are the distinguished tori in this case.

Suppose that $G^F$ is of type {\bf ${{}^2A_{2m-1}}$}. In this case $\SU_{2m}(q)\leq L \leq \GU_{2m}(q)$. Let $v$ be an anisotropic vector in $V$, and consider the stabilizer $H$ of $\langle v \rangle$ in $L$. The group $H$ contains a subgroup isomorphic to $\GU_{2m-1}(q)$ which acts faithfully by isometries on $\langle v\rangle^\perp$. Now, as in the previous paragraph, this subgroup contains a cyclic subgroup $X$ that acts irreducibly on $\langle v\rangle^\perp$. Taking $X$ to be of maximal order amongst such subgroups, and defining $g$ and $h$ as before, we conclude that $h$ is regular and lies in a unique maximal torus $T^F$ of $G^F$. The conjugates of $T^F$ are the distinguished tori in this case.

Suppose that $G^F$ is of type {\bf $B_n$}. In this case, again, $L$ is a group of isometries. We let $v$ be an anisotropic vector in $V$ such that $\langle v\rangle^\perp$ is a space of type ${\rm O}_{2n}^-$. Now the analysis proceeds as in the previous paragraph.

Suppose lastly that $G^F$ is of type {\bf $D_n$}. In this case we let $W$ be an anisotropic subspace of dimension $2$ and observe that $W^\perp$ is a space of type ${\rm O}_{2n-2}^-$. Now $L$ contains a subgroup $H$ isomorphic to $\Omega_2^-(q)\times \Omega_{2n-2}^-(q)$. We choose $X$ to be the direct product of a cyclic group $H_1$ of maximal order in $H$ acting irreducibly on $W$ and trivially on $W^\perp$, and a cyclic group $H_2$ of maximal order in $H$ acting irreducibly on $W^\perp$ and trivially on $W$.  Now choose $g$ so that the projection onto $H_1$ generates $H_1$ and the projection onto $H_2$ generates $H_2$. Then $g$ has distinct eigenvalues; its projection (or lift) $h$ in $G^F$ is regular, and lies in a unique maximal torus $T^F$ of $G^F$. The conjugates of $T^F$ are the distinguished tori in this case.

The following lemma will be essential.
\begin{lem}\label{l: classical}
 Let $G^F$ be a classical group of Lie type of dimension $d$, and let $T^F$ be a distinguished torus in $G^F$.
 \begin{enumerate}
  \item If $G^F$ is of type $A_n$, $C_n$, ${^2D_n}$ or ${{}^2A_{2m}}$, then let $t$ be a primitive prime divisor of $q^d-1$,
  \item If $G^F$ is of type ${{}^2A_{n}}$ or $B_n$, then let $t$ be a primitive prime divisor of $q^{d-1}-1$,
  \item If $G^F$ is of type $D_n$, then let $t$ be a primitive prime divisor of $q^{d-2}-1$,
 \end{enumerate}
provided, in each case, that a primitive prime divisor exists. Then in all cases where $t$ is defined, the torus $T^F$ contains an element of order $t$.
\end{lem}

Note that in light of the restrictions on $n$ listed in Table~\ref{t: weyl}, Zsigmondy's theorem implies that $t$ is defined in all cases except possibly when $G^F$ is of type $A_1$, or when $q=2$.

\begin{proof}
 The lemma follows directly from our definition of a distinguished torus, together with the list of cardinalities of Singer cycles found in \cite[Table 1]{bereczky}.
\end{proof}

\section{Lemmas on non-nilpotent sets}\label{s: groups}

In this section~${H}$ is a group and~${c}$ is an element of~${\longintegers}$. We begin with some results which are useful for giving upper
and lower bounds for~${\omegac{H}}$. The first result, which is based on \cite[Lemma 2.4]{ap}, requires us first to
extend the definition of ${\omega_c}$ to subsets of groups: for ${X\subseteq H}$,
we write~${\omegac{X}}$ for the size of the largest non-${c}$-nilpotent subset of~${X}$.

\begin{lem}\label{lem: upper bound}
 Let ${Y_1, \dots, Y_k}$ be subsets of~${H}$ such that ${Y_1\cup\cdots \cup Y_k=H}$. Then
\[ \omegac{H}\leq \omegac{Y_1}+\cdots +\omegac{Y_k}. \]
\end{lem}

\begin{proof}
Suppose that~${\Omega}$ is a non-${c}$-nilpotent set in~${H}$. Then clearly ${|\Omega \cap Y_i|\leq \omegac{Y_i}}$
for ${i=1,\dots, k}$, and the result follows.
\end{proof}

\begin{lem}\label{lem: lower bound}
Let ${X_1,\dots, X_k, Y_1, \dots, Y_k}$ be subsets of~${H}$. Suppose that the following hold:
\begin{enumerate}
\item for ${i=1,\dots, k}$, ${X_i}$ is a non-${c}$-nilpotent subset of~${Y_i}$ of maximum order,
\item if ${1\leq i<j\leq k}$ and ${x_i\in X_i, x_j\in X_j}$, then ${\langle x_i, x_j\rangle}$ is not ${c}$-nilpotent.
\end{enumerate}
Then ${X_1\cup \cdots\cup X_k}$ is a non-${c}$-nilpotent subset of~${H}$, and
\[ \omegac{H}\geq \omegac{Y_1}+\cdots  + \omegac{Y_k}. \]
\end{lem}
\begin{proof}
The fact that ${X_1\cup \cdots \cup X_k}$ is a non-${c}$-nilpotent subset of~${H}$ follows immediately from our suppositions, and
the lower bound stated for ${\omegac{H}}$ is an obvious consequence.
\end{proof}

As mentioned in the introduction, Lemmas~\ref{lem: upper bound} and~\ref{lem: lower bound} put into a more general form the
connection between nilpotent covers and non-nilpotent sets already estab\-lished in Proposition~\ref{p: N suffices}.
These lemmas have a number of useful special cases: we refer in particular to \cite[Lemma 3.2]{azad} and \cite[Lemma 3.3]{azad}.

The following result is a slightly refined version of the latter. For a prime ${p}$, we write~${\nu_p(H)}$ for the number of Sylow ${p}$-subgroups of
a finite group~${H}$.

\begin{lem}\label{lem: sylow}
Let~${H}$ be finite and let~${p}$ be a prime number dividing~${|H|}$. Let the distinct Sylow $p$-subgroups of~$H$ be
${P_1, P_2, \ldots, P_{\nu_p(H)}}$.  Suppose that
\begin{equation}\label{e: sylow}
P_1\setminus\bigcup_{i=2}^{\nu_p(H)} P_i\neq \emptyset.
\end{equation}
Then the set ${\{P_1,\dots, P_{\nu_p(H)}\}}$ is ${2}$-minimal, and there exists a non-nilpotent set ${\Omega\subseteq H}$
such that all elements of~${\Omega}$ are ${p}$-elements, and such that ${|\Omega|=\nu_p(H)}$.
\end{lem}

\begin{proof}
Note first that the condition \eqref{e: sylow} is equivalent to the assertion that~${H}$ contains elements which lie
in a unique Sylow ${p}$-subgroup of~${H}$. Since the Sylow ${p}$-subgroups of~${H}$ are conjugate, we see that every
Sylow ${p}$-subgroup ${P_i}$ of~${H}$ contains an element ${w_i}$ which lies in no other Sylow
${p}$-subgroup. Let~${\Omega}$ be the set ${\{w_i \mid i=1,\dots,\nu_p(H)\}}$. Then clearly ${|\Omega|=\nu_p(H)}$.

Let~${w_i}$ and~${w_j}$ be distinct elements of ${\Omega}$, and let ${W=\langle w_i, w_j\rangle}$.
Since no ${p}$\hbox{-}sub\-group of ${H}$ contains both ${w_i}$ and ${w_j}$, it follows that the Sylow ${p}$-subgroups of ${W}$ are not
normal in ${W}$, and so ${W}$ is not nilpotent. We conclude that~${\Omega}$ is a non-nilpotent set, and hence
that the family ${\{P_1,\dots, P_{\nu_p(H)}\}}$ is ${2}$-minimal.
\end{proof}

We will use Lemma~\ref{lem: sylow} in association with the following result.

\begin{lem}\label{lem: separate sets}
Let~${H}$ be finite and suppose that ${p_1, \dots, p_n}$ are distinct prime divisors of~${|H|}$.
For ${i=1,\dots, n}$, let~${\Omega_i}$ be a non-${c}$-nilpotent subset of~${H}$ such that the orders of the elements
of~${\Omega_i}$ are powers of~${p_i}$. Suppose that for ${w_i\in \Omega_i}$ and~${w_j\in\Omega_j}$,
we have ${[w_i, w_j]\neq 1}$ whenever ${i\neq j}$. Then
\[ \omegac{H} \geq |\Omega_1| + |\Omega_2| + \cdots + |\Omega_n|. \]
\end{lem}

\begin{proof}
The result follows easily from Lemma~\ref{lem: lower bound}, once we have made the following observation:
 if~${w_i}$ and~${w_j}$ are group elements of prime power order for distinct primes, then the subgroup
${\langle w_i, w_j\rangle}$ is nilpotent if and only if ${[w_i, w_j]=1}$.
\end{proof}

We end with two related results. The first result is a restated (and slightly generalized) version of
\cite[Lemma 3.1]{azad}, which we state here without further proof. The second is a useful corollary to the first.

\begin{lem}\label{lem: 1}
Let~${N}$ be a normal subgroup of~${H}$ and let~${\Omega}$ be a non-${c}$-nilpotent subset of~${H/N}$. Any set of representatives
of~${\Omega}$ in~${H}$ is a non-${c}$-nilpotent subset of~${H}$.
\end{lem}

\begin{lem}\label{lem: quasisimple}
Let~${Y}$ be a central subgroup of~${H}$. A set ${\Omega\subseteq H}$ is non-${c}$-nilpotent if and only if the set
\begin{align*}
\Omega/Y = \{hY \mid h\in \Omega\}
\end{align*}
is non-${c}$-nilpotent in~${H/Y}$ and ${|\Omega|=|\Omega/Y|}$.
\end{lem}

\begin{proof}
Suppose that~${\Omega}$ is non-${c}$-nilpotent. Suppose that ${g,h\in \Omega}$ are such that ${gY=hY}$. Then ${h=gy}$ for some
${y\in Y}$, and so ${\langle g, h\rangle = \langle g,y\rangle}$ is abelian; it follows that ${g=h}$, and we conclude
that ${|\Omega|=|\Omega/Y|}$. If ${\langle gY, hY\rangle}$ is ${c}$-nilpotent, then ${\langle g,h, Y\rangle}$ is ${c}$-nilpotent
(since ${Y}$ is central), and once again we see that ${g=h}$. So~${\Omega/Y}$ is non-${c}$-nilpotent, as required.

Now suppose that~${\Omega/Y}$ is non-${c}$-nilpotent and ${|\Omega|=|\Omega/Y|}$. Then~${\Omega}$ is a set of representatives
of~${\Omega/Y}$ in~${H}$ and the result follows from Lemma~\ref{lem: 1}.
\end{proof}

Lemma~\ref{lem: quasisimple} has an important consequence: it tells us that in order to study non-${c}$-nilpotent sets
in a finite simple group of Lie type~${J}$, it is sufficient to study them in~${L}$ where~${L}$ is any group such that
${L/Z(L)\cong J}$. Now Lemma~\ref{lem: versions} implies that there is a simple linear algebraic group~${G}$ and a
Frobenius endomorphism~${F}$ such that ${G^F/Z(G^F) \cong J}$; therefore it is sufficient to study non-${c}$-nilpotent sets
in the group~${G^F}$. For instance, to understand non-nilpotency in ${\PSL_2(q)}$, we can study non-${c}$-nilpotent sets in
${\SL_2(q)=G^F}$, where ${G=\SL_2(K)}$. Similarly, studying non-${c}$-nilpotent sets in~${\GL_n(q)}$ will tell us about
non-${c}$-nilpotent sets in~${\PGL_n(q) = G^F}$, where ${G=\PGL_n(K)}$.

In the reverse direction, Lemma~\ref{lem: quasisimple} implies that a knowledge of non-${c}$-nilpotent sets in finite simple groups
provides information on the non-${c}$-nilpotent sets in all quasi-simple groups.

\section{Some lower bounds}\label{s: lower}

Throughout this section,~${G}$ is a simple linear algebraic group over an algebraically closed field~${K}$ of
characteristic~${p}$, and~${F}$ is a Frobenius endomorphism for~${G}$.

\subsection{Non-nilpotent sets given by Sylow subgroups}

Our first result generalizes the main result
of \cite{azad}.

\begin{lem}\label{lem: regular unipotent set}
Let~${C}$ be a conjugacy class of\/~${G^F}$ such that elements of\/~${C}$ are~regular unipotent. There is a set
${\Omega\subseteq C}$ such that ${|\Omega|=\nu_p(G^F)}$ and~${\Omega}$ is non-nilpotent.
\end{lem}

\begin{proof}
We can construct such an~${\Omega}$ by including one element of~${C}$ from every Sylow ${p}$-subgroup of~${G^F}$. Since regular
unipotent elements lie in a unique Borel subgroup of~${G}$, they lie in a unique Sylow ${p}$-subgroup of~${G^F}$.
Now the result follows from Lemma~\ref{lem: sylow}.
\end{proof}

\begin{lem}\label{lem: regular semisimple set}
Suppose that ${C}$ is a conjugacy class of regular semisimple elements in~${G^F}$. Suppose also that the elements of~${C}$
have order~${t^a}$, where~${t}$ is a prime that does not divide~${|W|}$, and~${a}$ is a positive integer.
Then ${C}$ has a non-nilpotent subset ${\Omega}$ of size ${\nu_t(G^F)}$.
\end{lem}

\begin{proof}
Since~${t}$ does not divide~${|W|}$, any Sylow ${t}$-subgroup of~${G^F}$ lies inside a maximal torus of~${G^F}$. Since a regular
semisimple element lies in a unique maximal torus of~${G}$ (and hence a unique maximal torus of~${G^F}$), it follows that
each regular semisimple element of order~${t^a}$ lies in a unique Sylow ${t}$-subgroup of~${G^F}$. Now the result follows from
Lemma~\ref{lem: sylow}.
\end{proof}

\begin{thm}\label{thm: putting sets together}
Let ${T_1, \dots, T_k}$ be maximal tori in~${G^F}$, and suppose that they are pairwise non-conjugate. Suppose
that for ${i=1,\dots, k}$, the torus~${T_i}$ contains a~regular element~${g_i}$ of order~${t_i}$, where~${t_i}$ is a
prime that does not divide~${|W|}$. Then
\[ \omegan{G^F}\geq \nu_p(G^F)+\nu_{t_1}(G^F)+\cdots +\nu_{t_k}(G^F). \]
\end{thm}

\begin{proof}
Let~${\Omega_0}$ be the non-nilpotent set of regular unipotent elements given in Lemma~\ref{lem: regular
unipotent set}. Let~${\Omega_i}$ be the non-nilpotent set of regular semisimple elements in conjugates of~${T_i}$
given in Lemma~\ref{lem: regular semisimple set}. Let ${\Omega=\Omega_0\cup \Omega_1 \cup \cdots
\cup\Omega_k}$.

Let ${i\in\{1,\dots, n\}}$. Since~${t_i}$ does not divide~${|W|}$ we see that~${T_i}$ contains a Sylow ${t_i}$-subgroup
of~${G}$. Since~${g_i}$ is regular of order~${t_i}$, we know that it lies inside a unique maximal torus of~${G^F}$, and
hence ${t_i\neq t_j}$ for ${i\neq j}$. Furthermore, since~${g_i}$ is regular, we see that ${[g_i, h]\neq 1}$ for any
${h\not\in T_i}$. In particular ${[g_i, h]\neq 1}$ for all ${h\in \Omega\setminus\Omega_i}$.

Similarly, since an element ${g\in\Omega_0}$ is regular unipotent, if~${h}$ is a semisimple element of~${G}$ and
${[g,h]=1}$, then ${h\in Z(G)}$. Since ${t_i\nmid |W|}$, we see from Proposition \ref{prop: t divides W} that ${t_i}$ does not divide ${|Z(G)|}$, and hence that
${[g,h]\neq 1}$ for all ${g\in \Omega_0}$ and ${h\in\Omega\setminus\Omega_0}$. Now the result follows by Lemma
\ref{lem: separate sets}.
\end{proof}

\subsection{Rank dependent bounds for ${\omega_\infty(G^F)}$}

In this section we complete the proof of Theorem \ref{thm: rank-dependent lower bound}. To prove the lower bound when $G^F$ is a classical group, we will make use of the {\it distinguished tori} that we described in detail in \S\ref{s: classical}. If ${G^F}$ is an exceptional group of Lie type, then we take a different approach here, along lines indicated by \cite{bps}:

\begin{defn}
A non-trivial torus ${T}$ of the group ${G^F}$ is said to be \emph{sharp} if ${C_{G^F}(t) = T}$ for every non-identity ${t\in T}$.
\end{defn}

The next result is \cite[Theorem 3.1]{bps}.
\begin{prop}\label{p: sharp torus}
 All of the exceptional simple groups of Lie type except for ${E_7(q)}$ contain sharp maximal tori.
\end{prop}

We shall require two preliminary results.

\begin{lem}\label{l: torus bounds}
For every $r$, there exist positive constants $c_r, d_r$ and $e_r$ such that if $G$ is any simple algebraic group of rank $r$, $F$ is any Frobenius endomorphism of $G$, and $T$ is any $F$-stable maximal torus of $G$, then
\[
c_r q^r\leq {|T^F|\leq d_r q^r}\textrm{ and }{|G^F|\leq e_r q^{|\Phi|+r}}.
\]
\end{lem}
\begin{proof}
The upper bound on $|G^F|$ can be proved easily by checking the orders of the finite groups of Lie type. Now consider the problem of bounding $|T^F|$.

Recall first that $\mathcal{T}$ corresponds to an $F$-conjugacy class of the Weyl group $W$. Indeed if $w$ is an element of the $F$-conjugacy class, then $T=T_0^g$ where $T_0$ is a maximal split torus, $g^{-1}F(g)\in N_G(T_0)$ and $w$ is the corresponding element in the Weyl group $W\cong N_G(T_0)/T_0$. Let ${Y_0=\Hom(G_m,T_0)}$, the cocharacter group of  algebraic homo\-morphisms ${G_m(K) \to T_0}$. Then ${F}$ acts on ${Y_0}$ in a natural way: for ${\alpha \in Y_0}$ we have ${\alpha^F = F\circ\alpha}$. Since ${F}$ is a Frobenius map, there is a map
${F_0:Y_0 \to Y_0}$ such that ${F=[q]F_0}$, where ${[q]}$ is the map induced by the field automorphism~${x\mapsto x^q}$.

Now \cite[Proposition 3.3.5]{Carter} implies that
${|T^F|=\det_{Y_0 \otimes \R}([q]-F_0^{-1} w)}.$
Since ${F_0^{-1}w}$ has finite order, its eigenvalues ${m_1,\dots, m_r}$ on ${Y_0\otimes \C}$ are roots of unity. Thus
${[q]-F_0^{-1}w}$ has eigenvalues ${q-m_1,\dots, q-m_r}$, and hence
\[
|T^F| = (q-m_1)(q-m_2)\cdots (q-m_r).
\]
Now it is easy to see that
\[
c_rq^r \leq (q-1)^r \leq |T^F|\leq (q+1)^r\leq d_r q^r
\]
 where $c_r=(\frac{1}{2})^r$ and $d_r=(\frac{3}{2})^r$.
\end{proof}

\begin{prop}\label{p: maxtori}
Let ${r>1}$. There is a constant ${C_r}$ such that for any simple algebraic group ${G}$ of rank ${r}$, for any Frobenius endomorphism ${F}$ of ${G}$, and for any ${G^F}$\hbox{-}con\-jugacy class ${\mathcal{T}}$ of maximal tori of\/ ${G^F}$, we have ${N(G^F)\le C_r|\mathcal{T}|}$, where ${N(G^F)}$ is the number of ${F}$-stable maximal tori in ${G}$.
\end{prop}

\begin{proof}
For notational convenience in what follows, we shall write ${f(r)}$ as a token for an unspecified function of ${r}$, with the understanding that
separate instances of the token may refer to distinct functions.

Recall that, by Proposition~\ref{p: steinberg}, ${N(G^F) = q^{|\Phi|}}$ where ${\Phi}$ is the root system for ${G}$. If ${q}$ is bounded above by a function of ${r}$, then ${N(G^F)<f(r)}$ and the proposition holds trivially. Now \cite[Proposition 3.6.6]{Carter} implies that, for $q$ sufficiently large, all maximal tori of ${G^F}$ are non-degenerate, i.e. any maximal torus in $G^F$ lies in precisely one maximal torus of $G$. This implies, in particular, that ${N(G^F)}$ is equal to the number of maximal tori in ${G^F}$.

Since all maximal tori in ${G^F}$ are non-degenerate, \cite[Corollary 3.6.5]{Carter} implies that ${|N_{G^F}(T^F):T^F|<f(r)}$
for any maximal torus $T^F$ in $\mathcal{T}$. Thus, appealing to Lemma~\ref{l: torus bounds} (and using the constants $d_r$ and $e_r$ therein), we have
\[
\begin{aligned}
 |\mathcal{T}|=\frac{|G^F|}{|N_{G^F}(T^F)|}&\geq \frac{e_rq^{r+|\Phi|}}{|N_{G^F}(T^F)|} =N(G^F)\cdot \frac{e_rq^r}{|N_{G^F}(T^F)|}\\
 &\geq N(G^F)\cdot \frac{e_rq^r}{f(r)|T^F|} \geq N(G^F)\cdot \frac{e_r}{f(r)d_r}.
 \end{aligned}
\]
The result follows.
\end{proof}

\begin{proof}[Proof of Theorem~\ref{thm: rank-dependent lower bound}{\rm:} lower bound]
We begin by establishing the lower bound for ${\omega_\infty(G^F)}$. By Proposition \ref{p: maxtori} it suffices to show that \[ {\omega_\infty(G^F)\ge |\mathcal{T}|}, \] where ${\mathcal{T}}$ is some class of maximal tori.

Suppose that ${G^F}$ is an exceptional group, other than ${E_7(q)}$. Then ${G^F}$ has a~class~${\mathcal{T}}$ of sharp tori. Let ${T_1,T_2}$
be distinct tori in ${\mathcal{T}}$. Then clearly ${T_1\cap T_2 = \{1\}}$. But now for any non-identity ${x_1\in T_1}$ and ${x_2\in T_2}$ we see that
${\langle x_1,x_2\rangle}$ is centreless, and hence not nilpotent. It follows immediately that ${\omega_\infty(G^F)\ge|\mathcal{T}|}$.

Now suppose that ${G^F}$ is of classical type. Let ${T}$ be the distinguished torus identified in \S\ref{s: classical}, and let $d$ be the dimension of the group $G^F$. We may assume without loss of generality, that $q>2$; thus Lemma~\ref{l: classical} implies that, provided $G^F$ is not of type $A_1$, then $T$ contains an element ${x}$ of order ${t}$, where ${t}$ is a primitive prime divisor of ${q^d-1}$ or of~${q^{d-1}-1}$ or of~${q^{d-2}-1}$. If $G^F$ is of type $A_1$, then Theorems~\ref{thm: pgl2} and \ref{thm: psl2} imply the result, so we exclude this case and assume that $T$ contains the given element $x$.

By Lemma \ref{lem: fermat} we may suppose that ${t>(d-2)\log_p q}$, and, referring to Table~\ref{t: weyl}, we see that for all but finitely many values of ${q}$, ${t}$ is coprime with the order of the Weyl group of ${G}$. We may exclude these finite values of $q$ without loss of generality, hence, by Proposition \ref{p: springersteinberg}, we conclude that every Sylow ${t}$-subgroup of ${G^F}$ lies inside a conjugate of ${T}$.

Let ${\mathcal{T}}$ be the conjugacy class of ${T}$ and assume, first, that ${t}$ is a primitive prime divisor of ${q^d-1}$ or of~${q^{d-1}-1}$. Then $\langle x\rangle$ acts irreducibly on a space of dimension at least $d-1$ and we conclude immediately that $x$ is regular. It follows that the order of the intersection of any two distinct conjugates of ${T}$ is not divisible by ${t}$, and thus distinct Sylow ${t}$-subgroups of~${G^F}$ have trivial intersection. Let ${\mathcal{C}}$ be a set of elements of~${G^F}$ containing one element of order ${t}$ from each torus in~${\mathcal{T}}$. Then clearly any two elements of ${\mathcal{C}}$ generate a non-nilpotent group. Therefore ${\omega_\infty(G^F)\ge |\mathcal{T}|}$, as required.

Next assume that $t$ is a primitive prime divisor of ${q^{d-2}-1}$. In this case we have $G^F=D_n(q)$, with $d=2n\geq 8$. In this case $x$ is not regular, however, using the fact that ${\rm O}_2^-(q)$ is dihedral one can confirm, once again, that the order of the intersection of any two distinct conjugates of ${T}$ is not divisible by ${t}$. Now the argument proceeds as per the previous paragraph.

It remains only to deal with the groups ${E_7(q)}$. Recall that there are two versions of such a group, one simple and one not, but we use the same label for each since the same argument works. We use two facts, both established in the proof of \cite[Theorem 3.4]{bps}: First, that ${E_7(q)}$ contains a class ${\mathcal{T}}$ of maximal tori of order ${\frac{1}{\delta}(q^7-1)}$ where ${\delta\in\{1,2\}}$; second, that for ${T\in\mathcal{T}}$, for ${t}$ a primitive prime divisor of ${q^7-1}$, and for ${g\in T}$ an element of order ${t}$, the centralizer of ${g}$ in ${E_7(q)}$ is ${T}$. Now ${t}$ does not divide ${|E_7(q):T|}$, and it follows that the Sylow ${t}$-subgroups of ${E_7(q)}$ intersect trivially, and that each is contained in a unique member of ${\mathcal{T}}$. It is clear from these facts that ${\omega_\infty(E_7(q))\ge |\mathcal{T}|}$, as~required.
\end{proof}
It is perhaps worth mentioning that for groups ${G}$ of a fixed Lie type, it may well be possible to establish a rank-independent lower bound for
${\omega_\infty(G)}$ by a more careful count of maximal tori. We have not attempted to carry out this project, however.

\begin{proof}[Proof of Theorem~\ref{thm: rank-dependent lower bound}{\rm:} upper bound]
We now turn to the task of proving an upper bound for ${\omega_\infty(G^F)}$. We revive the convention, used in the proof of Proposition~\ref{p: maxtori}, of writing ${f(r)}$ as a token for an unspecified function of ${r}$.

Observe first, that all semisimple elements of ${G^F}$ are contained in a maximal torus of ${G^F}$; these are abelian, and by definition there are at most ${N(G^F)}$ of them. On the other hand, a result of Steinberg \cite[Theorem 6.6.1]{Carter} states that there are ${N(G^F)}$ unipotent elements in ${G^F}$. So taking the set of maximal tori of ${G^F}$ together with the set of cyclic groups generated by a unipotent element,
we have at most ${2N(G^F)}$ abelian subgroups of ${G^F}$, covering all unipotent and semisimple elements of ${G^F}$.

We therefore need only construct a family of nilpotent subgroups of ${G^F}$, of size at most ${f(r)N(G^F)}$, whose members together contain all the elements of ${G^F}$ that are neither unipotent nor semisimple. Let ${g}$ be such an element. Write ${g=su}$ for the Jordan decomposition of ${g}$; so ${s\in G^F}$ is semisimple, ${u\in G^F}$ is unipotent, and ${[s,u]=1}$. Note that ${u}$ lies in the centralizer of ${s}$; indeed, by~\cite[Pro\-po\-sition~14.7]{malletest}, we have that ${u}$ belongs to ${C_G(s)^0}$, the connected component of~${C_G(s)}$.

Now \cite[Corollary 14.3]{malletest} implies that there are in ${G}$ at most ${f(r)}$~${G}$-conjugacy classes of centralizers of semisimple elements. Furthermore \cite[\S1]{carter2} implies that if ${s\in G^F}$, then ${C_G(s)}^0$ is an ${F}$-stable reductive subgroup of maximal rank containing $s$, and moreover, that the number of ${G^F}$-conjugates of such centralizers is at most~${f(r)}$.

It will therefore be sufficient to prove that if ${C}$ is an ${F}$-stable subgroup of maximal rank, equal to ${C_G(s)^0}$ for some ${s\in G^F}$, then the unipotent elements of the ${G^F}$-conjugates of ${C^F\hspace*{-0.2mm}}$ can be covered by a set of ${p}$-subgroups of size~$\hspace*{-0.2mm}{f(r)N(G^F)}$. But the number of subgroups needed is at most
\[
\begin{split}
|\{G^F\textrm{-conjugates of }C^F\}| \cdot |\{\textrm{Sylow }p\textrm{-subgroups of }C^F\}| &\leq \frac{|G^F|}{|C^F|} \cdot  \frac{|C^F|}{|B^F|}
\\&=  \frac{|G^F|}{|B^F|},
\end{split}
\]
where ${B}$ is a Borel subgroup of ${C}$. In particular ${B}$ contains a maximal torus~${T}$ of~${G}$. Now Lemma~\ref{l: torus bounds} implies that ${|T^F|\geq f(r)q^r}$, and that ${|G^F|\leq f(r) q^{|\Phi|+r}}$. The result now follows.
\end{proof}

\section{Exact results for low rank groups}\label{s: exact}

In this section we calculate the value  of ${\omega_c(G)}$ for every~${c\in \longintegers}$, and for every $G$ in one of a number of families of Lie type of low rank. Some of these statistics have been calculated previously \cite{azad, aamz}.
In every case we establish the exact value of ${\omega_c(G)}$ by constructing a ${2}$-minimal ${c}$-nilpotent covering of~${G}$.

The groups we consider belong to the families ${A_1(q)}$, ${{}^2\!A_2(q)}$, ${{}^2\reduce B_2(2^{2m+1})}$ and ${{}^2G_2(3^{2m+1})}$. For the family
${A_1(q)}$, when ${q}$ is odd, there are two isogeny types to deal with, namely the groups ${\SL_2(q)}$ and ${\PGL_2(q)}$. Similarly, the family
${{}^2\!A_2(q)}$ comprises the groups ${\SU_3(q)}$ and ${\PGU_3(q)}$, which are non-isomorphic whenever ${q\equiv -1\pmod 3}$.

Before proceeding, we require a definition. Let ${\mathcal{P}=\{X_1, \dots, X_k\}}$ be a set of subgroups of a group~${G}$.
We say that ${\mathcal{P}}$ is a {\it partition} of~${G}$ if ${G=X_1\cup\cdots \cup X_k}$, and if ${X_i\cap X_j=\{1\}}$
for all ${1\leq i<j\leq k}$.

We will need the following preliminary result of Suzuki \cite{suzuki4}.

\begin{prop}\label{p: partition}
Let~${G}$ be an almost simple group admitting a partition~${\mathcal{P}}$. Then~${G}$ is one of\/ ${\PSL_2(q)}$ with ${q>3}$,
${\PGL_2(q)}$ with ${q>3}$, or ${{{}^2\reduce B_2(2^{2m+1})}}$ with ${m>0}$. Furthermore each group has a partition of form
${\mathcal{P}_T \cup \mathcal{P}_p}$, where~${\mathcal{P}_T}$ is the set of maximal tori in~${G}$, and~${\mathcal{P}_p}$
is the set of Sylow ${p}$-subgroups of~${G}$.
\end{prop}

Note that, by Proposition~\ref{p: steinberg}, we have ${|\mathcal{P}_T| = q^{|\Phi|}=(|G|_p)^2}$.
Several of the arguments in this part of the paper fail for certain small values of~${q}$, and in some of these cases we have relied on
direct computation. In every such case, we have established the value of~${\omega_c(G)}$ by finding a ${2}$-minimal ${c}$-nilpotent cover for ${G}$.
These calculations were performed using the Magma computer algebra package \cite{Magma}.

\subsection{${{\mathrm{PGL_2(q)}}}$}

The main result in this section is the following.

\begin{thm}\label{thm: pgl2}
Let ${G=\PGL_2(q)}$, and let ${c\in\longintegers}$. Then~${G}$ has a ${2}$-minimal ${c}$-nilpotent cover, and
\[ \omega_c(G)=\begin{cases}
4  & \textrm{if\/ ${q=2}$,} \\
7 & \textrm{if\/ ${c>1}$ and ${q=3}$,} \\
10 & \textrm{if\/ ${c=1}$ and ${q=3}$,} \\
q^2+q+1 & \textrm{if\/ ${q\geq 4}$.}
\end{cases} \]
\end{thm}

Before we prove Theorem~\ref{thm: pgl2} we state a lemma, which will be useful also in the next section when
we study the groups ${\SL_2(q)}$.

\begin{lem}\label{l: nilpotent in pgl2}
Let~${X}$ be a nilpotent subgroup of\/ ${\PGL_2(q)}$. Then one of the following holds.
\begin{enumerate}
\item ${X}$ is cyclic.
\item ${X}$ is elementary abelian.
\item ${X}$ is a dihedral ${2}$-group and $p$ is odd.
\end{enumerate}
\end{lem}

\begin{proof}
It is well known that ${\PGL_2(2)\cong S_3}$ and that ${\PGL_2(3)\cong S_4}$; in these cases it is easy to verify the result
directly. For ${q>3}$ we refer to \cite{king}, which gives a list of the maximal subgroups of
${\PGL_2(q)}$ and ${\PSL_2(q)}$; the result can readily be derived from this information.
\end{proof}

\begin{proof}[Proof of Theorem~\ref{thm: pgl2}]
When ${q=2}$ or ${q=3}$, the result can be calculated directly. We therefore suppose that ${q>3}$. Let
${\mathcal{P}=\mathcal{P}_T\cup\mathcal{P}_p}$ be the partition given in Proposition~\ref{p: partition}, and observe that
\[
{ |\mathcal{P}| = |\mathcal{P}_T|+|\mathcal{P}_p| = q^2+q+1}.
\]

The subgroups in ${\mathcal{P}}$ are abelian, and so by Lemma~\ref{lem: upper bound}, we have
${\omega_1(G)\leq |\mathcal{P}|}$. In each member of~${\mathcal{P}_T}$, we choose any element of order greater than~${2}$ to be a distinguished element; suitable elements exist since ${q>3}$, and by construction they are regular. In each member of ${\mathcal{P}_p}$, we choose any nontrivial element to be a distinguished element.

Let~${g}$ and~${h}$ be two distinct distinguished elements. We need to check that ${\langle g, h \rangle}$
is not nilpotent. If $p$ is odd, then Lemma~\ref{l: nilpotent in pgl2} implies that all nilpotent subgroups of ${\PGL_2(q)}$ are abelian, so it suffices to check that~${g}$ and~${h}$ do not commute. Indeed Lemma~\ref{l: nilpotent in pgl2} yields the same conclusion when $p$ is odd, since in this case both~${g}$ and~${h}$ have order greater than~${2}$.

Observe that the centralizer of~${g}$ is precisely that member
of~${\mathcal{P}}$ which contains~${g}$, and that this subgroup does not contain ${h}$. So ${\langle g, h\rangle}$ is not
nilpotent, and we see from Lemma~\ref{lem: lower bound} that~${\omega_\infty(G)\geq |\mathcal{P}|}$.
Now the result follows, since ${\omega_1(G)\geq \omega_\infty(G)}$.
\end{proof}

\subsection{${{\SL_2(q)}}$}

Consider the situation when ${G=\SL_2(q)}$. In the case that ${q}$ is even we have ${\SL_2(q)\cong \PGL_2(q)}$, and the result of the previous section applies. Thus in this section we assume that~${q}$ is odd. Our results generalize the main result of \cite{azad}.

\begin{thm}\label{thm: psl2}
Let ${q}$ be odd, and let ${G=\SL_2(q)}$. For every ${c\in\longintegers}$, the group~${G}$ has a ${2}$-minimal ${c}$-nilpotent cover,
and
\[ \omega_c(G)=\begin{cases}
5 & \textrm{if\/ ${q=3}$,} \\
21 & \textrm{if\/ ${q=5}$,}   \\
q^2+q+1 & \textrm{if\/ ${q>5}$.}                               \end{cases}
 \]
\end{thm}

\begin{proof}
If ${q=3}$ or ${5}$, we calculate the value of ${\omega_c(G)}$ directly. Assume that ${q>5}$ and let ${H=G/Z(G)\cong \PSL_2(q)}$. Let
${\mathcal{P}=\mathcal{P}_T\cup\mathcal{P}_p}$ be the partition of ${H}$ given in Proposition~\ref{p: partition}; we observe that
${|\mathcal{P}| = |\mathcal{P}_T|+|\mathcal{P}_p| = q^2+q+1}$.
For each member, ${\hat{\,}N}$, of ${\mathcal{P}}$, define ${N}$ to be the unique subgroup of ${G}$ that contains ${Z(G)}$ and satisfies ${N/Z(G)=\hat{\,}N}$. Let ${\mathcal{N}}$ be the set of all such ${N}$ and observe that ${|\mathcal{N}|=|\mathcal{P}|=q^2+q+1}$, that ${\mathcal{N}}$ covers ${G}$, that all members of ${\mathcal{N}}$ are abelian, and that the intersection of any two members of ${\mathcal{N}}$ equals ${Z(G)}$.

Now in every member ${N}$ of ${\mathcal{N}}$ we choose an element ${g}$ such that ${gZ(G)\in H}$ has order greater than ${2}$. From here the proof is identical to that of Theorem~\ref{thm: pgl2}.
\end{proof}

\subsection{${{\SU_3(q)}}$}\label{s:SU3}

Our main result in this section is the following.
\begin{thm}\label{thm: su}
 Let ${G=\SU_3(q)}$. For every ${c\in\longintegers}$, the group~${G}$ has a ${2}$-minimal ${c}$-nilpotent cover,
and
\[ \omega_c(G)= \begin{cases}
31 & \textrm{if\/ ${c\ge 2}$ and ${q=2}$,} \\
10 & \textrm{if\/ ${c=1}$ and ${q=2}$,} \\
757 & \textrm{if\/ ${c\ge 2}$ and ${q=3}$,} \\
q^6 + q^5 + q^3 + q^2 + 1 & \textrm{if\/ ${c\ge 2}$ and ${q> 3}$,} \\
q^6 + q^5 + \frac{1}{\delta}q^4 + \frac{1}{\delta}q^3 + q^2 + \frac{1}{\delta}q + \frac{1}{\delta} & \textrm{if\/ ${c=1}$ and ${q > 2}$,}
\end{cases} \]
where $\delta=|Z(G)|$.
\end{thm}

Recall that $|Z(G)|$ is equal to the greatest common divisor of $3$ and $q+1$.

\begin{proof}
When ${q\leq 3}$, the number ${w_c(G)}$ can be computed directly. We assume therefore that ${q>3}$ and we construct a ${2}$-minimal ${c}$-nilpotent cover of ${G}$. In what follows we write $Z$ for $Z(G)$.

{\bf Covering unipotent elements.}
Let ${U}$ be a Sylow ${p}$-subgroup of ${G}$, where ${q=p^a}$ for some positive integer ${a}$. Then $|Q|=q^3$ and $|U^G| = q^3+1$. On the other hand an easy calculation confirms that $G$ contains precisely $(q^3+1)(q^3-1)$ non-trivial unipotent elements. We conclude immediately that any two distinct Sylow $p$-subgroups of $G$ intersect non-trivially. In particular if $g\in U\backslash\{1\}$, and $N$ is a nilpotent group containing $g$, then the Sylow $p$-subgroup of $N$ is a subgroup of $U$.

Suppose first that ${g\in Z(U)}$. Then ${C_G(g) = U\rtimes C_{q+1}}$, and we let ${h}$ be an element of ${C_G(g)}$ of order ${q+1}$; one can check that \[{L:= C_{G}(h)\cong \GU_2(q)\cong \SL_2(q)\rtimes C_{q+1}}.\] Now ${g}$ lies in a unique Sylow ${p}$-subgroup of ${L}$ and we conclude that \[{M:= C_L(g)\cong Z(U)\times C_{q+1}}\] is the only maximal nilpotent subgroup containing ${gh}$ (in fact ${M}$ is abelian). Since ${Z(U)}$ is a characteristic subgroup of ${M}$, we conclude that \[{N_G(M)\leq N_G(Z(U))= U\rtimes C_{q^2-1}}.\] It is easy to check that ${N_G(M) = Z(U)\rtimes C_{q^2-1}}$ and so $|M^G| = q^2(q^3+1)$.

Suppose next that ${g\in U\backslash Z(U)}$. Then ${C_G(g)=U_0\times Z}$, where ${U_0}$ is an abelian subgroup of ${U}$ of order ${q^2}$. We conclude that any nilpotent subgroup containing ${g}$ must equal ${Q\times Z(G)}$, where ${Q}$ is a ${p}$-group. Since $U$ is the only Sylow $p$-subgroup containing $g$, we conclude that $U\times Z(G)$ is the unique maximal nilpotent subgroup of $G$ that contains $g$.

Observe that ${U\times Z}$ has nilpotency class ${2}$ and, since ${N_G(U)=U\rtimes C_{q^2-1}}$, we see that
  \[ |(U\times Z)^G=|U^G| = q^3+1. \]
 Similarly it is clear that ${C_G(g') = U_0\times Z}$ for all elements ${g'}$ of ${U_0\backslash Z(U)}$. It follows, in particular, that ${U_0\times Z}$ is the only maximal abelian subgroup of ${G}$ containing these elements. Since ${N_G(U_0) = U\rtimes C_{\delta(q-1)}}$, we conclude that
 \[ |(U_0\times Z)^G|=|U_0^G| = \frac{1}{\delta}(q^4 + q^3 + q + 1). \]

{\bf Covering mixed elements.} Recall that a \emph{mixed} element of $G$ is one that is neither unipotent nor semisimple. If $g$ is one such, then ${g=su}$, where ${s}$ is semisimple and ${u}$ is unipotent, and where ${[s,u]=1}$. By the above remarks, we obtain that either $s\in Z$, or else $u$ lies in the centre of some conjugate of $U$ and ${s\in C_G(u)}$. It follows that ${g}$ lies in some conjugate of ${M}$ or some conjugate of $U_0\times Z$.

{\bf Covering semisimple elements.}
We are left with the task of finding a set of ${c}$-nilpotent subgroups of ${G}$ containing all semisimple elements of ${G}$. Every semisimple element of ${G}$ lies in a maximal torus of ${G}$; each of these tori is isomorphic to one of ${C_{q^2-1}}$, ${C_{q+1}\times C_{q+1}}$ or~${C_{q^2-q+1}}$. We consider these in turn.

{\bf Tori isomorphic to ${C_{q^2-q+1}}$.}
Let $T_2$ be a maximal torus of $G$ isomorphic to $C_{q^2-q+1}$. Since ${q>2}$, Zsigmondy's theorem implies that there exists a primitive prime divisor ${t}$ of ${q^6-1}$. Immediately from its definition, we see that ${t}$ divides ${q^2-q+1}$, and is coprime to ${q^2-1}$. This implies, in particular, that $T_2$ contains a Sylow ${t}$-subgroup of ${G}$. Moreover, if ${g}$ is a generator of this Sylow ${t}$-subgroup, then ${g}$ is regular, i.e.\ ${C_G(g)=T_2}$; so any nilpotent subgroup of ${G}$ containing ${g}$ must be a subgroup of ${T_2}$.

{\bf Tori isomorphic to $C_{q^2-1}$.}
Let ${T_0}$ be a maximal torus of $G$ isomorphic to ${C_{q^2-1}}$. Suppose, first, that ${q-1}$ is divisible by an odd prime ${t}$. Then ${t}$ does not divide ${|G|/|T_0|}$ and so ${T_0}$ contains a Sylow ${t}$-subgroup of ${G}$. Moreover, any element ${g}$ of order~${t}$ must be regular, i.e.\ ${C_G(g)=T_0}$. This implies, in particular, that any nilpotent subgroup of ${G}$ containing ${g}$ is a subgroup of ${T_0}$.

Suppose, on the other hand, that ${q-1}$ is not divisible by an odd prime; so ${q-1=2^b}$ for some positive integer~${b}$. Then ${N_G(T_0)}$ is a nilpotent group containing a Sylow ${2}$-subgroup ${S}$ of ${G}$. Now ${S}$ has order $4(q-1)$ and $S\cap T_0$ is cyclic of order $2(q-1)$; furthermore any element in $S\backslash (S\cap T_0)$ has order $2$ or $4$. We set $g$ to be a generator of $S\cap T_0$ and observe that $g$ is regular and has order at least $8$. This implies, in particular, that any nilpotent subgroup of ${G}$ containing ${g}$ is a
subgroup of ${ST_0=N_G(T_0)}$, and so ${\langle g,g'\rangle}$ is
not nilpotent for any conjugate ${g'}$ of ${g}$ which is not contained in ${T_0}$.

{\bf Tori isomorphic to ${C_{q+1}\times C_{q+1}}$.}
Let ${T_1}$ be a maximal torus isomorphic to ${C_{q+1}\times C_{q+1}}$. Suppose that ${q+1}$ is divisible by a prime ${t>3}$. Then, just as in the previous case, it follows that ${T_1}$ contains a Sylow ${t}$-subgroup of ${G}$, and a regular element ${g}$ of order ${t}$, such that any nilpotent subgroup of ${G}$ containing ${g}$ is a subgroup of ${T_1}$.

Suppose that ${q+1}$ is divisible by ${9}$. Then ${T_1}$ contains a regular element ${g}$ of order ${9}$ and ${C_G(g)=T_1}$. On the other hand observe that if ${S}$ is a Sylow ${3}$\hbox{-}sub\-group of ${G}$, then ${S}$ is a subgroup of the normalizer of some conjugate of ${T_1}$, and all elements of order ${9}$ lie in that conjugate. In this case we conclude that any nilpotent subgroup of ${G}$ that contains ${g}$ lies in ${N_G(T_1)}$. In particular, if ${g'}$ is any conjugate of ${g}$ that does not lie in ${T_1}$, then ${\langle g, g'\rangle}$ is not nilpotent.

Suppose next that ${q+1}$ is divisible by ${12}$. Then ${T_1}$ contains a regular element ${g}$ of order ${12}$ such that both ${g^3}$ and ${g^4}$ are regular, i.e.\
\[
C_G(g)=C_G(g^3)=C_G(g^4)=T_1.
 \]
 This implies, in particular that if ${N}$ is a nilpotent subgroup of ${G}$ containing ${g}$, then all odd order subgroups of ${N}$ lie in ${T_1}$, and all subgroups of order coprime to ${3}$ lie in ${T_1}$; in particular ${N}$ itself is a subgroup of ${T_1}$. Again we conclude that, if ${g'}$ is any conjugate of ${g}$ that does not lie in ${T_1}$, then ${\langle g, g'\rangle}$ is not nilpotent.

Assume next that ${q+1=6}$. One can verify directly that if ${g}$ is a regular element of order ${6}$ in ${T_1}$, and if ${g'}$ is any ${G}$-conjugate of ${g}$, then ${\langle g,g'\rangle}$ is nilpotent if and only if ${g'\in T_1}$.

Assume, finally, that ${q+1=2^a}$ for some ${a}$, and let ${g}$ be a regular element in ${T_1}$; then ${C_G(g)=T_1}$ is a ${2}$-group, and so any nilpotent subgroup of ${G}$ that contains ${g}$ must lie in some Sylow ${2}$-subgroup ${S}$ of ${G}$. By comparing orders, we see that ${S}$ contains a conjugate of ${T_1}$ as a subgroup of index ${2}$; without loss of generality, we assume that ${S}$ contains ${T_1}$. Let ${h\in S\backslash T_1}$ and let ${\lambda_1, \lambda_2, \lambda_3}$ be the eigenvalues of ${h}$. An easy calculation implies that, up to relabelling, ${\lambda_2=-\lambda_1}$. Let ${\zeta}$ be a generator of the cyclic subgroup of ${\mathbb{F}_{q^2}}$ of order ${q+1}$. We may take~${g}$ to be an element with eigenvalues ${\zeta, \zeta^{\frac{q-1}{2}}, -1}$. Observe that ${\det g = 1}$, that the eigenvalues are distinct (so ${g}$ is regular), and that none of these eigenvalues is equal to ${-1}$ times any of the others. This implies that if ${S'}$ is any Sylow ${2}$-subgroup of ${G}$ containing ${g}$,
then ${g}$ lies in the unique maximal torus of ${G}$ contained as
an index ${2}$ subgroup in ${S}$. Since ${g}$ therefore lies in a unique maximal torus of ${G}$, namely ${T_1}$, we conclude that any nilpotent subgroup of ${G}$ that contains ${g}$ must be a subgroup of one of the three Sylow ${2}$-subgroups of ${G}$ that normalize ${T_1}$.

{\bf The size of the cover.}
Our calculations now yield a ${2}$-minimal ${c}$-nilpotent cover, ${\mathcal{N}}$, of ${G}$. If ${c\geq 2}$, set
\[ \mathcal{N} = (U\times Z)^G \cup M^G \cup T_0^G \cup T_1^G \cup T_2^G. \]
If ${c=1}$, set
\[ \mathcal{N} = (U_0\times Z)^G \cup M^G \cup T_0^G \cup T_1^G \cup T_2^G. \]
We have already observed that ${U_0^G\cup M^G}$ (and hence ${U^G\cup M^G}$) contains every non-semisimple element of ${G}$, while it is well~known that the set of all maximal tori, $T_0^G \cup T_1^G \cup T_2^G$, between them contain all semisimple elements of ${G}$. Thus~${\mathcal{N}}$ is, in each case, a ${c}$-nilpotent cover of ${G}$.

On the other hand our calculations imply that each member ${N}$ of ${\mathcal{N}}$ contains an element ${g}$ such that any ${c}$-nilpotent subgroup containing ${g}$ must lie in ${N}$ (or, in a couple of exceptional cases, in a small overgroup of ${N}$). Furthermore no member of ${\mathcal{N}}$ is a subgroup of any other member of ${\mathcal{N}}$ (or, in the exceptional cases, a subgroup of the relevant overgroups of any other member of ${\mathcal{N}}$). Hence ${\mathcal{N}}$ is 2-minimal, as required.

In order to calculate the order of ${\mathcal{N}}$ recall that ${|T_0^G\cup T_1^G \cup T_2^G| = q^6}$ by Proposition~\ref{p: steinberg}. The result follows.
\end{proof}

\subsection{${{\mathrm{\PGU_3(q)}}}$}

In this section we assume that ${q\equiv -1\pmod 3}$, since otherwise ${\PGU_3(q)\cong \PSU_3(q)}$ and the results of \S\ref{s:SU3}, combined with Lemma~\ref{lem: quasisimple}, yield the value of ${\omega_c(G)}$.

\begin{thm}\label{thm: pgu}
 Let ${G=\PGU_3(q)}$ with ${q\equiv -1\pmod 3}$. For every $c\in\longintegers$, the group~${G}$ has a ${2}$-minimal ${c}$-nilpotent cover,
and
\[ \omega_c(G)=\begin{cases}
49 & \textrm{if\/ ${c\ge 2}$ and ${q=2}$,} \\
71 & \textrm{if\/ ${c=1}$ and ${q=2}$,} \\
q^6 + q^5 + q^3 + q^2 + 1 & \textrm{if\/ ${c\ge 2}$ and ${q> 2}$,} \\
q^6 + q^5 + q^4 + q^3 + q^2 + q + 1 & \textrm{if\/ ${c=1}$ and ${q> 2}$.}
\end{cases} \]
 \end{thm}
\begin{proof}
For ${q=2}$ we compute the value of ${\omega_c(G)}$ directly. For ${q>2}$ our argument is virtually identical to that of \S\ref{s:SU3}, and so we give only a brief summary here. We define a family ${\mathcal{N}}$ of subgroups of ${G}$ as follows. If ${c\geq 2}$, set
\[ \mathcal{N} = (\,\hat{\,}U)^G \cup M_0^G \cup \mathcal{T}. \]
If ${c=1}$, set
\[ \mathcal{N} = (\,\hat{\,}U_0)^G \cup M_0^G \cup \mathcal{T}. \]
Here ${\mathcal{T}}$ is the set of maximal tori of ${G}$; these come in three conjugacy classes, just as in ${\SU_3(q)}$. The group ${\hat{\,}U}$ is a Sylow ${p}$-subgroup of ${G}$ while ${\hat{\,}U_0}$ is the subgroup of ${{\PSU}(3,q)}$ equal to the projection of the group ${U_0}$ given in ${\SU_3(q)}$.

Finally one constructs ${M_0}$ by taking ${g\in Z(\,\hat{\,}U)}$. Then
${C_G(g)=U\rtimes C_{q+1}}$. Let ${C}$ be a cyclic subgroup of ${C_G(g)}$ of order ${q+1}$. Then define
${M_0:=  Z(\,\hat{\,}U)\times C}$.
 To show now that ${\mathcal{N}}$ is a 2-minimal ${c}$-nilpotent cover of ${G}$, one follows the line of argument of \S\ref{s:SU3}.
\end{proof}

\subsection{${{{}^2\reduce B_2(2^{2m+1})}}$}

In this section we deal with the Suzuki groups. Our main result generalizes \cite[Theorem 1.2]{aamz} which gives the exact value for ${\omega_1(G)}$ in the case that $G\cong{{{}^2\reduce B_2(2^{2m+1})}}$ for some $m\geq 1$.

It will be convenient to redefine the variable ${q}$, which hitherto has been defined as the level of the Frobenius endomorphism ${F}$. For the group
${{{}^2\reduce B_2(2^{2m+1})}}$ this would give the fractional power ${q=2^{m+\frac12}}$. For a clearer exposition in this section, we shall take ${q}$ to be the square of that value, i.e.\ ${q=2^{2m+1}}$.

\begin{thm}\label{thm: Sz} Let ${q=2^{2m+1}}$ for ${m\geq 0}$, and let ${G}$ be the Suzuki group ${{{}^2\reduce B_2}(q)}$. Then, for all ${c\in\longintegers}$, the group~${G}$ has a ${2}$-minimal
${c}$-nilpotent cover. Furthermore, if ${m\geq 1}$, then
\[ \omega_c(G)= \begin{cases}
q^4+q^2+1 & \textrm{if\/ ${c\ge 2}$,} \\
q^4+q^3-q^2+q-1 & \textrm{if\/ ${c=1}$.}
\end{cases} \]
If ${m=0}$, then ${\omega_c(G)=6}$ for all ${c\in\longintegers}$.
\end{thm}

\begin{proof}
If ${m=0}$, then ${{{}^2\reduce B_2}(q)\cong C_5\rtimes C_4}$ and the result can be computed directly. Assume ${m>0}$ and let ${\mathcal{P}=\mathcal{P}_T\cup\mathcal{P}_p}$ be the partition given in Proposition~\ref{p: partition}. The value of ${|\mathcal{P}_T|}$ is given by Proposition~\ref{p: steinberg} as ${q^4}$ (taking into account the modified definition of ${q}$ currently in force). The value of ${|\mathcal{P}_p|}$ is given by \cite[Theorem~9]{suzuki2} as ${q^2+1}$. Thus
\[ |\mathcal{P}| = |\mathcal{P}_T|+|\mathcal{P}_p| = q^4+q^2+1. \]
The members of~${\mathcal{P}_T}$ are abelian, and the members of~${\mathcal{P}_p}$ are nilpotent of class~${2}$.
So we have ${\omega_2(G)\leq |\mathcal{P}|}$, by Lemma~\ref{lem: upper bound}.

Every member $H$ of~${\mathcal{P}}$ is a maximal nilpotent subgroup of~${G}$. For each ${g\hspace*{-0.2mm}\in\hspace*{-0.2mm} H{\setminus}\{1\}}$,
the only maximal nilpotent subgroup containing~${g}$ is~${H}$ itself \cite[Theorem 9]{suzuki2}. It follows from
Lemma~\ref{lem: lower bound} that ${\omega_\infty(G)\geq |\mathcal{P}|}$, and this is enough to establish the first part of the theorem.

We are left with the task of calculating~${\omega_1(G)}$. Let ${P}$ be a Sylow
${p}$-subgroup of~${G}$, and let ${g\in P\setminus Z(P)}$. Then the centralizer of ${g}$ in ${G}$ is ${\langle g, Z(P)\rangle}$.
If~we choose ${g_1,\dots, g_{q-1}}$ to be representatives from the ${q{-}1}$ non-trivial cosets
of~${Z(P)}$ in~${P}$, then we see that ${[g_i, g_j]\neq 1}$ whenever ${1\leq i<j\leq q-1}$. Define
${H_i := \langle g_i, Z(P)\rangle}$
for ${i=1,\dots, q-1}$; then the subgroups ${H_i}$ form a ${2}$-minimal abelian cover of~${P}$, with the elements ${g_i}$ as distinguished elements.
We may repeat this process for every member of ${\mathcal{P}_p}$. The resulting abelian subgroups, together with
those of ${\mathcal{P}_T}$, form a ${2}$-minimal abelian cover of~${G}$. The size of this cover is given by
\[ |\mathcal{P}_T| + (q-1)|\mathcal{P}_p| = q^4 + (q-1)(q^2+1), \]
and the result follows.
\end{proof}

\subsection{${{{}^2G_2(3^{2m+1})}}$}

The final family of groups to deal with is that of the Ree groups of type ${{}^2G_2(3^{2m+1})}$. As with the Suzuki groups, it is convenient here to redefine ${q}$ to be the square of its value in earlier sections. Thus, in what follows, we set ${q=3^{2m+1}}$.

We will make extensive use of the structural
information about the group ${{^2G_2}(q)}$ given by the main theorem of \cite{ward}; and we
rely on the classification of the maximal subgroups of ${{^2G_2}(q)}$ provided by \cite[Theorem C]{kleidman2}.
Facts stated without proof in this section are taken from the statements of these two theorems.

We also make use of the following parametrization of a Sylow ${3}$-subgroup~${P}$ of~${{^2G_2}(q)}$, which can be found in \cite{atlas}. We write elements of~${P}$ as triples ${(x,y,z) \in \Fq^3}$ with the multiplication
given by
${(x_1,y_1,z_1)\cdot (x_2, y_2, z_2)}$ equal to
\begin{equation}\label{e: multiply}
 (x_1+x_2, y_1+y_2 + x_1x_2^s - x_1^sx_2, z_1+z_2 +
y_1x_2 + x_1^sx_2^2 + x_1x_2^{s+1} - x_1^2x_2^s),
\end{equation}
where ${s=3^{m+1}}$. Under this parametrization it is not hard to see that
\[
\begin{split}
Z(P) &=\{(x,y,z)  \mid  x=y=0\}, \\
Z_2(P) &= \{(x,y,z)  \mid  x=0\},
\end{split}
\]
where~${Z_2(P)}$ is the second term of the upper central series for~${P}$. Both of these groups are elementary
abelian, with orders~${q}$ and~${q^2}$ respectively.

\begin{thm}\label{thm: Ree}
Let ${q=3^{2m+1}}$ for ${m\geq 0}$, and let ${G}$ be the Ree group ${{^2G_2}(q)}$. Then, for all ${c\in\longintegers}$, the group~${G}$ has a ${2}$-minimal
${c}$-nilpotent cover. Furthermore
\[ \omega_c(G)=
\begin{cases}
316 & \textrm{if\/ ${c\ge 2}$ and ${q=3}$,}\\
372 & \textrm{if\/ ${c=1}$ and ${q=3}$,}\\
q^6+q^5+q^3+q^2+1 & \textrm{if\/ ${c\geq 3}$ and ${q>3}$,} \\
q^6 + q^5 + \frac12q^4-\frac12q^3+q^2+\frac12q-\frac12 & \textrm{if\/ ${c=2}$ and ${q>3}$,} \\
q^6+\frac32q^5-\frac12q^4+\frac32q^2-\frac12q & \textrm{if\/ ${c=1}$ and ${q>3}$.}
\end{cases} \]
\end{thm}
\begin{proof}
In the case that ${m=0}$, and hence ${q=3}$, we have the isomorphism
${{^2G_2}(q)\cong \PSL_2(8).3}$,
and the result can be computed directly. Assume that ${m>0}$ and observe
that the set of prime divisors of~${|G|}$ can be partitioned naturally into
six sets:
\begin{enumerate}
\item the two singleton sets~${\{2\}}$ and~${\{3\}}$,
\item the two sets ${\mathcal{R}^\pm}$ of prime divisors of ${ q\pm\sqrt{3q}+1}$,
\item the set ${\mathcal{S}}$ of prime divisors of ${\frac{q-1}{2}}$,
\item  the set ${\mathcal{T}}$ of prime divisors ${\frac{q+1}{4}}$.
\end{enumerate}
We say that an element ${g}$ of ${G}$ \emph{belongs} to one of the sets ${\mathcal{R}^\pm}$, ${\mathcal{S}}$ or ${\mathcal{T}}$ if all of the
prime divisors of the order of ${g}$ belong to that set. If the element ${g}$ belongs to any of these sets, then there is a unique maximal nilpotent subgroup
${M(g)}$ containing~${g}$, which is equal to the centralizer in ${G}$ of ${g}$.
\begin{enumerate}
\item If ${g}$ belongs to ${\mathcal{R}^\pm}$, then ${M(g)}$ is cyclic of order ${q\pm\sqrt{3q} +1}$. We
write~${C_\pm}$ for the subgroup ${M(g)}$.
\item If ${g}$ belongs to ${\mathcal{S}}$, then ${M(g)}$ is cyclic of order ${q-1}$. We
write~${D}$ for the subgroup ${M(g)}$ in this case.
\item If~${g}$ belongs to ${\mathcal{T}}$, then ${M(g)}$ is isomorphic to ${C_{\frac{q+1}{2}}\times C_2}$.
We write~${E}$ for the subgroup ${M(g)}$.
\end{enumerate}

Define ${\mathcal{N}_T}$ to be the set of all conjugates in ${G}$ of ${C_\pm}$, ${D}$ or ${E}$.
So ${\mathcal{N}_T}$ is a~set of abelian subgroups of ${G}$, and every element of ${G}$ whose order is divisible by a prime greater
than~${3}$ lies in a member of ${\mathcal{N}_T}$. Furthermore if~${g}$ and~${h}$ are
elements of maximal order in distinct members of ${\mathcal{N}_T}$, then ${\langle g, h\rangle}$ is clearly not nilpotent.

It remains to deal with elements of~${G}$ whose order is not divisible by a prime greater than ${3}$; these have order~${2}$,~${3}$,~${6}$ or~${9}$.
Let~${P}$ be a Sylow ${3}$-subgroup of~${G}$ and let~${g_P}$ be an element in~${P}$ of the maximal order~${9}$. Since ${P\cap Q=\{1\}}$ for
distinct Sylow ${3}$-subgroups of~${G}$, and since~${P}$ contains the centralizer of~${g}$, we see that~${P}$ is the only maximal
nilpotent subgroup containing~${g}$.

Note next that there is a unique conjugacy class of involutions in~${G}$; it follows easily from the
centralizer structure that we have outlined, by straightforward Sylow arguments,
that ${\mathcal{N}_T\cup P^G}$ contains all elements of~${G}$ of order other than~${6}$.

Now if ${g\in G}$ has order~${6}$, then the only maximal nilpotent subgroup containing~${g}$ is the
centralizer in~${G}$ of~${g}$, which we shall call ${F}$. The subgroup ${F}$ is isomorphic to ${C_2\times A(q)}$,
where~${A(q)}$ is an elementary abelian group of order~${q}$.
We see that ${\mathcal{N}=\mathcal{N}_T\cup P^G\cup F^G}$ covers every element of~${G}$.

For every element~${N}$ of~${\mathcal{N}}$, we have seen that ${N}$ has an element~${g_N}$ such that~${N}$ is the only maximal
nilpotent subgroup containing~${g_N}$. It follows that~${\mathcal{N}}$ is a ${2}$-minimal nilpotent cover
of~${G}$. We note, furthermore, that every member of~${\mathcal{N}}$ is abelian, apart from the Sylow ${3}$-subgroups
of~${G}$. So~${\mathcal{N}}$ is a ${2}$-minimal ${3}$-nil\-potent cover of~${G}$.
Now Proposition~\ref{p: N suffices} implies that for ${c\geq 3}$, we have
\[ \omega_c(G)=|\mathcal{N}| = |\mathcal{N}_T|+|P^G|+|F^G|. \]

The set~${\mathcal{N}_T}$ is in one-to-one correspondence with the maximal tori of~${G}$ and thus, by
Proposition~\ref{p: steinberg}, it has order~${q^6}$. Thus
\begin{equation*}
\begin{aligned}
\omega_C(G) &= q^6 + |P^G|+|F^G| \\
&= q^6 + \frac{|G|}{q^3(q-1)} + \frac{|G|}{q(q-1)} \\
&=q^6+q^5+q^3+q^2+1.
\end{aligned}
\end{equation*}

We now consider the case ${c\hspace*{-0.2mm}=\hspace*{-0.2mm}2}$. We construct a ${2}$-minimal ${2}$-nilpotent cover~${\mathcal{N}_2}$ as follows.
From the set ${\mathcal{N}}$ constructed above, we retain the sets~${\mathcal{N}_T}$ and~${F^G}$, since the subgroups in
these sets are abelian. Let ${g\in P\setminus Z_2(P)}$ be an element of order~${9}$, and define\vspace*{1mm}
\[
P_g := \langle g, Z_2(P)\rangle. \vspace*{1mm}
\]
Using the parametrization of ${P}$ described above, we write
${g\hspace*{-0.2mm}=\hspace*{-0.2mm}(x_0,y_0,z_0)}$. We~see that\vspace*{1mm}
\[
 P_g=\{(x,y,z)  \mid  x\in \{0,x_0, 2x_0\},\, y,z\in\Fq\}. \vspace*{1mm}
\]
Now clearly ${Z(P_g) = Z(P)}$, and ${P_g/Z(P_g)}$ is abelian. We conclude that~${P_g}$ is a non-abelian group of
order~${3q^2}$ and nilpotency class~${2}$.

For ${i=1,\dots, \frac{q-1}{2}}$, let ${g_i=(x_i, y_i, z_i)}$ be elements of~${P}$, chosen in such a way that
${x_1,\dots, x_{\frac{q-1}{2}}\in \Fq}$
are pairwise linearly independent as vectors over~${\mathbb{F}_3}$,
i.e.\ so that the set of subspaces \[ \biggl\{\langle x_i \rangle\mid i=1,\dots, \frac{q-1}{2}\biggr\} \]
contains every~${1}$-subspace of~${\mathbb{F}_q}$. It is easy to see that the corresponding subgroups~${P_{g_i}}$ cover the Sylow
${3}$-group~${P}$; there are ${\frac{q-1}{2}}$ of these subgroups.

We claim that the groups~${P_{g_i}}$ are a ${2}$-minimal ${2}$-nilpotent cover of~${P}$. To see this we must show that
for distinct~${i}$ and~${j}$, the group ${P_{i,j}:= \langle g_i, g_j\rangle}$ has nilpotency class~${3}$. Using
\eqref{e: multiply}, one can check that ${Z(P_{i,j}) = P_{i,j}\cap Z(P)}$; and since ${P_{i,j}/ Z(P_{i,j})}$
is clearly non-abelian, we conclude that~${P_{i,j}}$ has nilpotency class ${3}$, as required.
The elements ${g_i}$ will be our distinguished elements.

The construction just described may be repeated inside any Sylow ${3}$-subgroup of~${G}$, to obtain in each a set of ${\frac{q-1}{2}}$
nilpotent subgroups of class ${2}$. Let~${\mathcal{N}_{2,p}}$ be the collection of all these subgroups. We recall that
${P \cap P^g=\{1\}}$ whenever~${P}$ and~${P^g}$ are distinct, and so any pair of distinguished elements from distinct
members of~${\mathcal{N}_{2,p}}$ generate either a non-nilpotent group (if they belong to distinctSylow ${3}$-subgroups),
or a group of nilpotency class~${3}$ (otherwise).

We now define
${\mathcal{N}_2:= \mathcal{N}_T \cup F^G \cup \mathcal{N}_{2,p}}$,
and observe that the members of
${\mathcal{N}_{2,p}}$ cover~${G}$.
We have seen that each member~${N}$ of~${\mathcal{N}_T \cup F^G}$ contains an element~${g_N}$ which lies in no other
maximal nilpotent subgroup of ${G}$. On the other hand, for each member~${N}$ of~${\mathcal{N}_{2,p}}$
there is an element~${g_N\in N}$ such that the only maximal nilpotent subgroup containing~${g_N}$ is a Sylow ${3}$-subgroup
of~${G}$; furthermore, if ${N_1,N_2\in\mathcal{N}_{2,p}}$ are distinct, then the elements ${g_{N_1}}$ and ${g_{N_2}}$
generate either a non-nilpotent group or a group of nilpotency class~${3}$.
It follows that~${\mathcal{N}_2}$ is a ${2}$-minimal ${2}$-nilpotent cover of~${G}$.

We now calculate that
\[
\begin{aligned}
|\mathcal{N}_2| &= |\mathcal{N}_T| + |F^G| + |\mathcal{N}_{2,p}|  \\
&= |\mathcal{N}_T| + |F^G| +  \frac12(q-1)|P^G| \\
&= q^6 + q^5+q^2 + \frac12(q-1)(q^3+1)
\end{aligned}
\]
and the claimed result for ${c=2}$ follows.

It remains to deal with the case ${c=1}$. We construct a ${2}$-minimal ${1}$-nilpotent (abelian) cover, and as a consequence obtain a
maximal non-commuting set in~${G}$. As we did for ${c=2}$, we retain the sets~${\mathcal{N}_T}$ and~${F^G}$ of abelian subgroups.

Let ${h=(x_0,y_0,z_0)\in P\setminus Z_2(P)}$ be an element of order~${9}$. Define
${Q_h := \langle h, Z(P)\rangle}$,
and observe that
\[ Q_h=\{(x,y,z)  \mid  (x,y)\in \{(0,0),(x_0, y_0), (2x_0, 2y_0)\},z\in\Fq\}. \]
Clearly~${Q_h}$ is abelian of order~${3q}$.

For ${i=1,\dots, \frac{q(q-1)}{2}}$, let ${h_i=(x_i, y_i, z_i)}$ be elements of ${P\setminus Z_2(P)}$, chosen so that
${(x_1, y_1),\dots, (x_{\frac{q-1}{2}}, y_{\frac{q-1}{2}})\in \Fq\times \Fq}$
are pairwise linearly independent as
vectors over~${\mathbb{F}_3}$. So the set \[ \biggl\{\langle (x_i,y_i) \rangle\mid i=1,\dots, \frac{q(q-1)}{2}\biggr\} \]
contains all of those~${1}$-spaces in ${\mathbb{F}_q\times \mathbb{F}_q}$ with non-zero first coordinate. It is easy to see that
the set of corresponding subgroups~${Q_{h_i}}$ covers all elements of~${P\setminus Z_2(P)}$, as well
as all elements of~${Z(P)}$. There are ${\frac{q(q-1)}{2}}$ subgroups in this set.
Using~\eqref{e: multiply}, it is clear that for distinct~${i}$ and~${j}$, we have ${[h_i, h_j]\neq 1}$, and so~${\langle h_i, h_j\rangle}$ is non-abelian.

As in the case ${c=2}$, we can repeat the construction inside every Sylow ${3}$-sub\-group of~${G}$, obtaining in each case
a collection of ${\frac{q-1}{2}}$ abelian subgroups. Let~${\mathcal{N}_{1,p}}$ be the set of all the subgroups so obtained, and let
\[ \mathcal{N}_1= \mathcal{N}_T \cup F^G \cup \mathcal{N}_{1,p}. \]

We have seen that ${\mathcal{N}_T\cup F^G}$ contains all elements of~${G}$ whose order is not a power of~${3}$.
Now~${\mathcal{N}_{1,p}}$ contains all~${3}$-elements, apart from those conjugate to an element of ${Z_2(P)\setminus Z(P)}$.
But such an element is contained in a member of~${F^G}$, and so we
conclude that~${\mathcal{N}_1}$ covers~${G}$.

We must now establish the ${2}$-minimality of ${\mathcal{N}_2}$. We have seen how to find distinguished elements of the subgroups in
${\mathcal{N}_T \cup F^G}$. Every member~${N}$
of~${\mathcal{N}_{1,p}}$ contains an element~${g_N}$ such that the only maximal nilpotent subgroup containing~${g_N}$
is a Sylow ${3}$-subgroup of~${G}$; furthermore, any two such elements generate a non-abelian group. It follows that~${\mathcal{N}_1}$
is a ${2}$-minimal abelian cover of~${G}$.

Finally, we observe that
\[
\begin{aligned}
|\mathcal{N}_1| &= |\mathcal{N}_T| + |F^G| + |\mathcal{N}_{1,p}|  \\
&= |\mathcal{N}_T| + |F^G| +  \frac12q(q-1)|P^G| \\
&= q^6 + q^5+q^2 + \frac12q(q-1)(q^3+1)
\end{aligned}
\]
and the result follows.
\end{proof}

\section{Questions and conjectures}\label{s: questions}

The main results of this paper suggest a number of interesting questions which we discuss below.

\subsection{Questions about exact formulae}

Let ${G}$ be a simple algebraic group, and ${G^F}$ a finite group of Lie type with (twisted) Lie rank equal to ${1}$.
This is the situation treated in \S\ref{s: exact} above, and our results from that section, collectively, have a number of suggestive properties.
\begin{itemize}
 \item Provided that ${q\hspace*{-0.3mm}>\hspace*{-0.3mm}5}$, the isogeny class of ${G}$ does not affect the value of ${\omegan{G^F}}$.
 \item Provided that ${q>5}$, the value of ${\omegan{G^F}}$ is a polynomial in ${q}$. Furthermore, all of the coefficients in this polynomial are equal to ${1}$ or ${-1}$.
 \item Provided that ${q>3}$, the values of ${\omega_\infty(\SU_3(q))}$ and ${\omega_\infty({^2G_2}(q))}$ coincide.
\end{itemize}

We are naturally interested in how far these phenomena generalize. For instance, let us continue to suppose that ${G^F}$ is a finite group of Lie type, where ${G}$ is simple, but let us drop the assumption that ${G^F}$ has rank ${1}$. Then we have the following questions.
\begin{question}
Is ${\omegan{G^F}}$ independent of the isogeny class of ${G}$ for large enough~${q}$?
\end{question}
\begin{question}
Is ${\omegan{G^F}}$ a polynomial in ${q}$ for large enough ${q}$? Are the coefficients of ${\omega_{\infty}(G^F)}$ always equal to ${1}$ or ${-1}$?
\end{question}

\begin{question}
Can we classify (and explain) any coincidences in the value of ${\omega_c(G^F)}$ for different families of groups ${G^F}$?
\end{question}

\subsection{Intersection with conjugacy classes}

In this paper we have focused on the problem of calculating the order of a maximal non-nilpotent set in a finite group $G$. We are also interested in the possible structure of such a set and an obvious first approach is to study the interaction of non-nilpotent sets with conjugacy classes. With this in mind, then, we define a non-trivial conjugacy class which is a non-nilpotent set to be a {\it non-nilpotent class} of the finite group $G$. We pose the following question:

\begin{question}
Suppose that ${G}$ contains a non-nilpotent class. Can we describe the structure of ${G}$?
\end{question}

A natural first question concerning the structure of such a group ${G}$ would be to ask if it can be simple: we conjecture below that this is impossible. Some evidence for this conjecture can be found in the following result. Recall that ${O_p(G)}$ is the largest normal ${p}$-group in ${G}$.

\begin{lem}
Suppose that ${G}$ contains a non-nilpotent class of elements of prime order. Then ${G}$ is not simple.
\end{lem}
\begin{proof}
Let ${C}$ be a non-nilpotent class of elements of prime order ${p}$ in ${G}$ and let~${x,y\in C}$. The group ${\langle x, y \rangle}$ will be nilpotent precisely if it is a ${p}$-group. Thus the class ${C}$ will be non-nilpotent if and only if no two elements of ${C}$ lie in the same Sylow ${p}$-subgroup of ${G}$. But in this case the elements of ${C}$ are {\it isolated} in the Sylow ${p}$-subgroup in which they lie. Now the ${Z_p^*}$ Theorem (due to Glauberman \cite{glauberman} when ${p=2}$, and to Guralnick and Robinson \cite{gurrob}
for odd ${p}$) implies that ${C\subset O_p(G)}$, and hence that the group is not simple.
\end{proof}

So much for negative information about groups containing non-nilpotent conjugacy classes. On the other hand there are certainly many groups containing such a class. (In what follows, for ${x}$ an element of a group ${H}$, we write ${x^H}$ for the conjugacy class of ${H}$ containing ${x}$.)

\begin{lem}
Suppose that ${G}$ is a Frobenius group with complement ${H}$, and let ${x\in Z(H)}$. Then ${x^G}$ is a non-nilpotent set in ${G}$.
\end{lem}

Note that, in particular, a Frobenius group with abelian complement contains a conjugacy class which is a non-nilpotent set. A relevant example is ${\PSL_2(3) \cong A_4}$ in which both conjugacy classes of elements of order ${3}$ are non-nilpotent classes.

\begin{proof}
Write ${K}$ for the Frobenius kernel of ${G}$ and note that ${(|K|, |H|)=1}$. Let~${y}$ be a conjugate of ${x}$. If ${\langle x,y\rangle \cap K=\{1\}}$, then ${x}$ and ${y}$ lie in the same complement which is a contradiction. Hence ${\langle x, y \rangle \cap K\neq \{1\}}$. But ${C_G(x)\cap K =\{1\}}$ and ${(o(x), |K|)=1}$, thus ${\langle x,y\rangle}$ is not nilpotent.
\end{proof}

Rather than restricting our attention to non-nilpotent classes, one can ask the more general question of how conjugacy classes of a group interact with non-nilpotent sets. In this respect, a classical theorem of Baer and Suzuki is relevant (see \cite[p.\,298]{huppertI}, \cite{suzuki2}).

\begin{thm}
 Let ${C}$ be a conjugacy class in ${G}$ such that for all ${x,y\in G}$, ${\langle x, y \rangle}$~is nilpotent. Then ${C\subset F(G)}$, the Fitting group of ${G}$.
\end{thm}

In particular this result implies that, in a simple group, every conjugacy class contains a non-nilpotent set of size at least ${2}$. Other results of this ilk---focussed on the property of solvability rather than nilpotency---can be found in~\cite{dghp}. It~is unclear, however how much the value ${2}$ can be increased.

\begin{question}
Let ${G}$ be a simple group. Can one state a minimum bound for the quantity
\begin{equation}\label{e: conjugacy class}
\min\{\omega_\infty(C)  \mid  C \textrm{ is a non-trivial conjugacy class in } G\}?
\end{equation}
\end{question}

One might also ask about upper bounds for the quantity \eqref{e: conjugacy class} for a simple group~${G}$. In this regard we posit the following conjecture, the truth of which would imply, in particular, that a simple group cannot contain a non-nilpotent class.

\begin{conj}
 Let ${G}$ be a simple group. Then
 \begin{equation}
\max\biggl\{\frac{\omega_\infty(C)}{|C|}  \mid  C \textrm{ is a non-trivial conjugacy class in } G\biggr\} \leq \frac12.
\end{equation}
\end{conj}

It is clear that the bound ${\frac 12}$ cannot be improved: in the group ${A_5}$, the class of elements of order ${3}$, and the two classes of elements of order ${5}$, all contain non-nilpotent sets of half their size. The same is true in ${\PSL_2(7)}$ of the conjugacy class of elements of order ${3}$.

We remark, finally, that one might ask similar questions for {\it non-commuting classes}. (We say that a conjugacy class ${C}$ in a group ${G}$ is non-commuting if, for all distinct ${g,h}$ in ${C}$, the group ${\langle g,h\rangle}$ is non-abelian.) In particular it is conceivable that the previous conjecture remains true in this more general context, i.e.\ with ${\omega_\infty}$ replaced in the statement by ${\omega_1}$.

\subsection{Rank and nilpotency class}

Let ${N}$ be a nilpotent subgroup of ${G}$, a finite group of Lie type. The connection between the nilpotency class of ${N}$ and the rank of ${G}$ appears to be slightly subtle.

We note, first of all, that the nilpotency class of ${N}$ cannot, in general, be bounded above by a function of the rank of ${G}$. One illustrative example is the case that ${N}$ is a dihedral ${2}$-subgroup of ${G=\PSL_2(p)}$, of the largest possible order ${2^k}$. The value of ${k}$ is unbounded as ${p}$ varies across the primes, and since~${N}$ has nilpotency class ${k-1}$, the non-existence of a bound in terms of the rank of ${G}$ is demonstrated.

On the other hand let ${\ell}$ be the minimum number such that ${\omega_i(G) = \omega_\ell(G)}$ for all ${i\geq \ell}$. Since ${G}$ is finite, such a number ${\ell}$ exists; but in light of the remarks of the previous paragraph, there is no {\it a priori} reason why ${\ell}$ should be bounded above by any function of the rank. It is nevertheless tempting to conjecture that such a bound exists.
\begin{conj}
For all ${r\in \mathbb{Z}^+}$, there exists ${\ell\in\mathbb{Z}^+}$ such that if ${G}$ is a finite group of Lie type of rank~${r}$, then ${\omega_i(G)=\omega_\ell(G)}$ for all ${i\geq \ell}$.
\end{conj}

Our results for rank ${1}$ groups provide some evidence for the plausibility of this conjecture. Note, in particular, that the presence of nilpotent
subgroups of unbounded class does not prevent the conjecture being true for the group ${\PSL_2(q)}$.

In another direction one might ask whether the example of the dihedral ${2}$\hbox{-}sub\-groups of ${\PSL_2(q)}$ is in some respects atypical. In particular, one could pose the following question.

\begin{question}\label{q: bound on nilpotency}
Is it the case that for all ${r\in \mathbb{Z}^+}$, there exists ${\ell\in\mathbb{Z}^+}$ such that if ${G}$ is a finite group of Lie type of rank ${r}$, and if ${N}$ is a nilpotent subgroup of ${G}$ of odd order, then the nilpotency class of ${N}$ is less than ${\ell}$?
\end{question}

An obvious way of addressing this question would be to undertake a detailed study of the maximal nilpotent subgroups of finite groups of Lie type. Interesting results in this direction already exist: for instance Vdovin has classified, for every finite simple group, the nilpotent subgroups of maximal order \cite{vdovin}.

\subsection{Non-solvable subsets}

A natural variant of the statistics we have studied in this paper replaces nilpotence with solvability. In this context we study the derived length of a subgroup rather than its nilpotency class. Let ${G}$ be a group and ${c}$ an element of ${\mathbb{Z}^+\cup\{\infty\}}$.
If~${c\in\mathbb{Z}^+}$, then we define ${G}$ to be {\it ${c}$-solvable} if~${G}$ is solvable with derived series of length~${c}$. We define ${G}$ to be {\it ${\infty}$-solvable} if ${G}$ is solvable.

A subset ${X}$ of ${G}$ is said to be {\it non-${c}$-solvable} if, for any two distinct elements ${x,y}$ in ${X}$, ${\langle x, y\rangle}$ is a subgroup of ${G}$ which is not ${c}$-solvable.
Define~${\beta_c(G)}$ to be the maximum order of a non-${c}$-solvable set in ${G}$. One can study the behaviour of ${\beta_c(G)}$ in much the same way as we have done here for ${\omega_c(G)}$, as well as defining the notion of a {\it ${2}$-minimal ${c}$-solvable cover} in the obvious way.

In the case that ${c=1}$, the two notions of ${c}$-solvability and ${c}$-nilpotency are identical. At the other end of the spectrum, however, for large ${c}$, one might expect significant difference in their behaviour. We might ask the same sort of question in this connection as Question~\ref{q: bound on nilpotency}
above:

\begin{question}\label{q: bound on solvability}
Is it the case that, for all ${r\in \mathbb{Z}^+}$, there exists ${\ell\in\mathbb{Z}^+}$ such that if~${G}$ is a finite group of Lie type of rank~${r}$, and if~${N}$ is a solvable subgroup of~${G}$, then the derived series of~${N}$ has length less than~${\ell}$?
\end{question}

\subsection*{Acknowledgments}
The third author was a frequent visitor to the University of Bristol while the work for
this paper was being undertaken; he would like to thank the maths department there for their hospitality.
Particular thanks are due to Jeremy Rickard for fruitful discussions on the subject of groups of Lie type.

All three authors would like to extend their thanks to an anonymous referee who read an earlier version of our paper very carefully and whose comments have improved the paper immeasurably. In particular this referee pointed out a number of small but significant errors, and we wish to register our thanks to them for finding these errors and thereby giving us the opportunity to fix them.

\bibliographystyle{plain}
\bibliography{paper2}

\begin{thebibliography}{10}

\bibitem{aamz}
A.~Abdollahi, A.~Azad, A.~Mohammadi~Hassanabadi, and M.~Zarrin.
\newblock On the clique numbers of non-commuting graphs of certain groups.
\newblock {\em Algebra Colloq.}, 17(4):611--620, 2010.

\bibitem{am}
A.~Abdollahi and A.~Mohammadi~Hassanabadi.
\newblock Finite groups with a certain number of elements pairwise generating a
  non-nilpotent subgroup.
\newblock {\em Bull.\ Iranian Math.\ Soc.}, 30(2):1--20, 99--100, 2004.

\bibitem{azad}
A.~Azad.
\newblock On nonnilpotent subsets in general linear groups.
\newblock {\em Bull.\ Aust.\ Math.\ Soc.}, 83(3):369--375, 2011.

\bibitem{aips}
A.~Azad, M.~A. Iranmanesh, C.~E. Praeger, and P.~Spiga.
\newblock Abelian coverings of finite general linear groups and an application
  to their non-commuting graphs.
\newblock {\em J. Algebraic Combin.}, 34(4):683--710, 2011.

\bibitem{ap}
A.~Azad and C.~E. Praeger.
\newblock Maximal subsets of pairwise noncommuting elements of
  three-dimensional general linear groups.
\newblock {\em Bull.\ Aust.\ Math.\ Soc.}, 80(1):91--104, 2009.

\bibitem{bps}
L.~Babai, P.~P. P{\'a}lfy, and J.~Saxl.
\newblock On the number of $p$-regular elements in finite simple groups.
\newblock {\em LMS J. Comput.\ Math.}, 12:82--119, 2009.

\bibitem{bang}
A.S. Bang.
\newblock Talteoretiske unders{\o}lgelser.
\newblock {\em Tidskrifft Math.}, 5(4):130--137, 1886.

\bibitem{bereczky}
{\'A}ron Bereczky.
\newblock Maximal overgroups of {S}inger elements in classical groups.
\newblock {\em J. Algebra}, 234(1):187--206, 2000.

\bibitem{blackburn}
S.~R. Blackburn.
\newblock Sets of permutations that generate the symmetric group pairwise.
\newblock {\em Combin.\ Theory Ser.\ A}, 113(7):1572--1581, 2006.

\bibitem{Magma}
W.~Bosma, J.~Cannon, and C.~Playoust.
\newblock {\em The {M}agma algebra system.\ {I}.\ {T}he user language}, 1997.

\bibitem{BEGHM}
J.~R. Britnell, A.~Evseev, R.~M. Guralnick, P.~E. Holmes, and A.~Mar\'oti.
\newblock Sets of elements that pairwise generate a linear group.
\newblock {\em J. Combin.\ Theory Ser.\ A}, 115(3):442--465, 2008.

\bibitem{bg}
J.~R. Britnell and N.~Gill.
\newblock Small nilpotent covers and large non-nilpotent subsets of
  $\mathrm{GL}_n(q)$.
\newblock 2012.\ In preparation.

\bibitem{brown1}
R.~Brown.
\newblock Minimal covers of {$S_n$} by abelian subgroups and maximal subsets of
  pairwise noncommuting elements.
\newblock {\em J. Combin.\ Theory Ser.\ A}, 49(2):294--307, 1988.

\bibitem{brown2}
R.~Brown.
\newblock Minimal covers of {$S_n$} by abelian subgroups and maximal subsets of
  pairwise noncommuting elements.\ {II}.
\newblock {\em J. Combin.\ Theory Ser.\ A}, 56(2):285--289, 1991.

\bibitem{carter2}
R.~W. Carter.
\newblock Centralizers of semisimple elements in finite groups of {L}ie type.
\newblock {\em Proc.\ London Math.\ Soc.\ (3)}, 37(3):491--507, 1978.

\bibitem{Carter}
R.~W. Carter.
\newblock {\em Finite groups of {L}ie type}.
\newblock Pure and Applied Mathematics (New York). John Wiley \& Sons Inc., New
  York, 1985.
\newblock Conjugacy classes and complex characters, A Wiley-Interscience
  Publication.

\bibitem{atlas}
J.~H. Conway, R.~T. Curtis, S.~P. Norton, R.~A. Parker, and R.~A. Wilson.
\newblock {\em Atlas of finite groups}.
\newblock Oxford University Press, 1985.

\bibitem{cr2}
C.~W. Curtis and I.~Reiner.
\newblock {\em Methods of representation theory.\ {V}ol.\ {II}}.
\newblock Pure and Applied Mathematics (New York). John Wiley \& Sons Inc., New
  York, 1987.
\newblock With applications to finite groups and orders, A Wiley-Interscience
  Publication.

\bibitem{dghp}
S.~Dolfi, R.~M. Guralnick, M.~Herzog, and C.~E. Praeger.
\newblock A new solvability criterion for finite groups.
\newblock {\em J. Lond.\ Math.\ Soc.\ (2)}, 85(2):269--281, 2012.

\bibitem{Endimioni}
G.~Endimioni.
\newblock Groupes finis satisfaisant la condition {$(\mathcal{N},n)$}.
\newblock {\em C. R. Acad.\ Sci.\ Paris S\'er.\ I Math.}, 319(12):1245--1247,
  1994.

\bibitem{glauberman}
G.~Glauberman.
\newblock Central elements in core-free groups.
\newblock {\em J. Algebra}, 4:403--420, 1966.

\bibitem{gls3}
D.~Gorenstein, R.~Lyons, and R.~Solomon.
\newblock {\em The classification of the finite simple groups.\ {N}umber 3.\
  {P}art~{I}.\ {C}hapter {A}, Almost simple $K$-groups}, volume~40 of {\em
  Mathematical Surveys and Monographs}.
\newblock American Mathematical Society, Providence, RI, 1998.

\bibitem{gurrob}
R.~M. Guralnick and G.~R. Robinson.
\newblock On extensions of the {B}aer-{S}uzuki theorem.
\newblock {\em Israel J. Math.}, 82(1-3):281--297, 1993.

\bibitem{huppertI}
B.~Huppert.
\newblock {\em Endliche {G}ruppen.\ {I}}.
\newblock Die Grundlehren der Mathematischen Wissenschaften, Band 134.
  Springer-Verlag, Berlin, 1967.

\bibitem{king}
O.~H. King.
\newblock The subgroup structure of finite classical groups in terms of
  geometric configurations.
\newblock In {\em Surveys in combinatorics 2005}, volume 327 of {\em London
  Math.\ Soc.\ Lecture Note Ser.}, pages 29--56. Cambridge Univ.\ Press,
  Cambridge, 2005.

\bibitem{kleidman2}
P.~B. Kleidman.
\newblock The maximal subgroups of the {C}hevalley groups ${G}_2(q)$ with $q$
  odd, the {R}ee groups ${^2G_2(q)}$ and their automorphism groups.
\newblock {\em J. Algebra}, 117(1):30--71, 1988.

\bibitem{malletest}
G.~Malle and D.~Testerman.
\newblock {\em Linear algebraic groups and finite groups of {L}ie type}, volume
  133 of {\em Cambridge Studies in Advanced Mathematics}.
\newblock Cambridge University Press, Cambridge, 2011.

\bibitem{neumann}
B.~H. Neumann.
\newblock A problem of {P}aul {E}rd{\H o}s on groups.
\newblock {\em J. Austral.\ Math.\ Soc.\ Ser.\ A}, 21(4):467--472, 1976.

\bibitem{pyber}
L.~Pyber.
\newblock The number of pairwise noncommuting elements and the index of the
  centre in a finite group.
\newblock {\em J. London Math.\ Soc.\ (2)}, 35(2):287--295, 1987.

\bibitem{springersteinberg}
T.~A. Springer and R.~Steinberg.
\newblock Conjugacy classes.
\newblock In {\em Seminar on {A}lgebraic {G}roups and {R}elated {F}inite
  {G}roups ({T}he {I}nstitute for {A}dvanced {S}tudy, {P}rinceton, {N}. {J}. ,
  1968/69)}, Lecture Notes in Mathematics, Vol.\ 131, pages 167--266. Springer,
  Berlin, 1970.

\bibitem{steinberg3}
R.~Steinberg.
\newblock Regular elements of semisimple algebraic groups.
\newblock {\em Inst.\ Hautes \'Etudes Sci.\ Publ.\ Math.}, 25:49--80, 1965.

\bibitem{suzuki4}
M.~Suzuki.
\newblock On a finite group with a partition.
\newblock {\em Arch. Math. (Basel)}, 12:241--254, 1961.

\bibitem{suzuki2}
M.~Suzuki.
\newblock On a class of doubly transitive groups.
\newblock {\em Ann.\ of Math.}, 75(1):105--145, 1962.

\bibitem{Tomkinson}
M.~J. Tomkinson.
\newblock Hypercentre-by-finite groups.
\newblock {\em Publ.\ Math.\ Debrecen}, 40(3-4):313--321, 1992.

\bibitem{vdovin}
E.~P. Vdovin.
\newblock Large nilpotent subgroups of finite simple groups.
\newblock {\em Algebra Log.}, 39(5):526--546, 630, 2000.

\bibitem{ward}
H.~N. Ward.
\newblock On {R}ee's series of simple groups.
\newblock {\em Trans.\ Amer.\ Math.\ Soc.}, 121:62--89, 1966.

\bibitem{zsig}
K.~Zsigmondy.
\newblock Zur {T}heorie der {P}otenzreste.
\newblock {\em Monatsh.\ Math.\ Phys.}, 3(1):265--284, 1892.

\end{thebibliography}

\end{document}